\documentclass[a4paper]{article}
\usepackage{amsmath,amssymb,mathtools,amsthm}
\usepackage{xcolor,graphicx,float,enumerate}

\usepackage[font=small]{caption}
\usepackage{subcaption}

\usepackage[hidelinks]{hyperref}
\usepackage[noabbrev,nameinlink]{cleveref}

\usepackage[utf8]{inputenc}
\DeclareUnicodeCharacter{00A0}{ }
\usepackage{anysize}

\newcommand{\ee}{\varepsilon}

\newcommand{\G}{\mathbb{G}}

\newcommand{\M}{\mathbb{M}}
\newcommand{\K}{\mathbb{K}}

\newcommand{\p}{\mathbb S}

\newcommand{\SFL}{(-\Delta)^s_{\mathrm{SFL}}}
\newcommand{\RFL}{(-\Delta)^s_{\mathrm{RFL}}}
\newcommand{\CFL}{\mathrm{CFL}}

\newcommand{\Ls}{\ensuremath{\mathrm{L}}}
\newcommand{\E}{\ensuremath{\mathrm{E}}}
\newcommand{\D}{\ensuremath{\mathrm{D}}}
\newcommand{\J}{\ensuremath{\mathrm{J}}}

\newcommand{\asympT}{\overset  t \asymp}

\newcommand{\Green}{\mathcal{G}}
\newcommand{\Heat}{\mathcal{H}}
\newcommand{\Martin}{\mathcal{M}}
\newcommand{\Semi}{\mathcal{S}}

\DeclareMathOperator{\esssup}{ess\,sup}
\DeclareMathOperator{\essinf}{ess\,inf}

\newcommand{\re}{\mathbb{R}}
\newcommand{\ren}{\mathbb{R}^n}

\DeclareMathOperator{\sign}{sign}

\newcommand{\angles}[1]{\left\langle#1\right\rangle}
\newcommand{\set}[1]{\left\{#1\right\}}
\newcommand{\abs}[1]{\left\lvert#1\right\rvert}
\newcommand{\norm}[2][]{\left\|#2\right\|_{#1}}

\newtheorem{proposition}{Proposition}[section]
\newtheorem{theorem}[proposition]{Theorem}

\newtheorem{lemma}[proposition]{Lemma}
\theoremstyle{definition}
\newtheorem{example}[proposition]{Example}
\newtheorem{remark}[proposition]{Remark}
\newtheorem*{theorem*}{Theorem}

\numberwithin{equation}{section}

\usepackage[
url=false,
isbn=false,
backref=bibtex,
backend=bibtex,
maxcitenames=4,
giveninits=true,
uniquename=init,
maxbibnames=100,
doi=false,
block=none]{biblatex}

\renewbibmacro{in:}{}
\DeclareFieldFormat[article]{citetitle}{#1}
\DeclareFieldFormat[article]{title}{#1}

\newcommand*\diff{\mathop{}\!\mathrm{d}}
\newcommand{\defeq}{\vcentcolon=}
\newcommand{\defbyeq}{=\vcentcolon}

\newcommand{\email}[1]{\href{mailto:#1}{#1}}

\allowdisplaybreaks[1]

\title{Singular solutions for fractional parabolic \\ boundary value problems}
\author{%
	Hardy Chan\thanks{Dept. of Mathematics, ETH Z\"urich.
	\email{hardy.chan@math.ethz.ch}}
\and
David G\'omez-Castro\thanks{Instituto de Matemática Interdisciplinar,
	Universidad Complutense de Madrid.
	\email{dgcastro@ucm.es}}
\and	
Juan Luis V\'{a}zquez\thanks{Depto. de Matem\'aticas
	Univ. Aut\'onoma de Madrid.
	\email{juanluis.vazquez@uam.es}}
}
\begin{document}
	\maketitle
	
\begin{abstract}
The standard problem for the classical heat equation posed in a bounded domain $\Omega$ of $\ren$  is the initial and boundary value problem. If the Laplace operator is replaced by a version of the fractional Laplacian, the initial and boundary value problem can still be solved on the condition that the non-zero boundary data must be singular, i.e., the solution $u(t,x)$ blows up as $x$ approaches $\partial \Omega$ in a definite way. In this paper we construct a theory of existence and uniqueness of solutions of the parabolic problem with singular data taken in a very precise sense, and also admitting  initial data and a forcing term. When the boundary data are zero we recover the standard fractional heat semigroup. A general class of integro-differential operators may replace the classical fractional Laplacian operators, thus enlarging the scope of the work.

As further results on the spectral theory of the fractional heat semigroup, we show that a  one-sided Weyl-type law holds in the general class, which was previously known for the restricted and spectral fractional Laplacians, but is new for the censored (or regional) fractional Laplacian. This yields bounds on the fractional heat kernel. \\ 


\end{abstract}

    \tableofcontents
\section{Introduction}

If we consider the classical heat equation posed in a bounded domain $\Omega$ of $\ren$ with  $n >2s$, a standard problem is the initial and boundary value problem with Dirichlet data. The theory goes in parallel with a similar theory for the elliptic Laplace--Poisson problem. When the Laplace operator is replaced by a version of the fractional Laplacian in the elliptic problem, Abatangelo and collaborators \cite{Abatangelo2015, Abatangelo2015b, Abatangelo2017a, abatangelo+gc+vazquez2019} made clear that standard boundary values are not accepted in the theory and must be replaced by singular values. Not only that, the singular values have to be specified in a very precise way. A well-posed problem follows in that case.

Two questions immediately arise from the elliptic studies, in particular \cite{abatangelo+gc+vazquez2019}: whether this singular behaviour is preserved in the parabolic problem, and what is the corresponding theory. We will try to give a satisfactory answer to those questions in this paper.

\subsection{General Evolution Problem}
 Let us state the evolution problem under study in its general formulation. First, we have an evolution equation
\begin{equation}	\label{eq:main}
		u_t + \Ls u = f{(t,x)}, \qquad \mbox{ for \ }  \   x\in\Omega,\  t\in(0,T),
\end{equation}
driven by an operator $\Ls$ belongoing to  a wide class that includes the usual fractional Laplacian versions. We also need  admissible boundary conditions to be satisfied by our solutions. They are formulated in terms of the boundary operator $\E$ that originated in the elliptic theory:
\begin{equation}	\label{eq:sing.bc}
\E u(t,\zeta) = h(t,\zeta), \qquad \mbox{ for \ }  \ \zeta \in \partial \Omega, t>0.
\end{equation}
We will examine the operator $\E $ below. Depending on the operator $\Ls$, we will need to impose an exterior boundary condition, which we will always take as homogeneous:
\begin{equation}	\label{eq:sing.ext}
u(t,x)  = 0, \qquad \mbox{ for \ }  \ x \in \Omega^c, \ t>0.
\end{equation}
Finally, we need to impose initial data as usual:
\begin{equation}	\label{eq:init}
u(0,x) = u_0(x), \qquad  \mbox{ for \ }  \ x \in \Omega.
\end{equation}

The form of the admissible boundary condition  is a consequence of the already established elliptic theory, since we want the solutions of the corresponding elliptic problems to appear as stationary solutions of the evolution problem. Consequently, we define the singular boundary condition as the limit
	\begin{equation}
	\label{eq:E definition}
		\E u (t,\zeta) \defeq \lim_{\substack{x \to \zeta \\ x \in \Omega}} \frac{u(t,x)}{u^\star (x)}
	\end{equation}
	where $u^\star$ is a particular known function, which typically exhibits a boundary behaviour of type $\delta(x)^{2s - \gamma  -1}$, where $\delta$ denotes the distance-to-the-boundary function $\delta(x) = \text{dist} (x, \Omega^c)$. Hereafter, the parameters $s\in(0,1)$ and $\gamma\in(0,1]$ determine respectively the interior and boundary behaviour of $\Ls$, see \eqref{eq:G estimate}.
	
	We call \eqref{eq:E definition} singular boundary data because of the values taken for the Restricted Fractional Laplacian (RFL) and Spectral Fractional Laplacian (SFL), where $2s - \gamma - 1 < 0$. However, for the Censored (or Regional) Fractional Laplacian (CFL), $2s - \gamma - 1 = 0$ and this is a Dirichlet type condition. As mentioned in \cite{abatangelo+gc+vazquez2019}, no examples are known to satisfy $2s-\gamma-1>0$. In this last setting, $\E u = 0$ seems to be a redundant condition $u = 0$ in $\partial \Omega$, which calls into question if the Green operator $\Green$,
	introduced in \eqref{greenG}, comes from a reasonable direct operator $\Ls$. Nevertheless, as in \cite{abatangelo+gc+vazquez2019} we will include a mathematical framework for these problems.

\subsection{Main assumptions, results and plan of the paper}

	We will denote the solution of  equation \eqref{eq:main} with initial-and-boundary conditions \eqref{eq:sing.bc}, \eqref{eq:sing.ext} and \eqref{eq:init}  by $\Heat[u_0, f, h]$ or, more precisely, by $\Heat_{\Ls} [u_0, f, h]$.
	Although it is a bulky notation, we have decided to use it for the sake of precision.

	We will prove that under general assumptions on the Green's operator $\Green$ for the elliptic problem
$\Ls v = f$, we  can construct a well-posed theory of  so-called  {\sl weak-dual solutions} for the parabolic problem.
	As in \cite{abatangelo+gc+vazquez2019}, we assume the following hypothesis which will be made precise as we introduce the involved quantities:
	\begin{itemize}
		\item  \eqref{eq:G estimate}: two-sided estimates for the elliptic Green's function
		\item \eqref{eq:G delta Linf to delta C}: boundary regularity of the elliptic Green's operator
		\item \eqref{eq:S submarkovian}: the elliptic differential operator generates a submarkovian semigroup (see all three in \Cref{sec.ellprl})
		\item \eqref{eq:Green normal}: existence of $\gamma$-normal derivative of the elliptic Green's function (see \Cref{ssec.32})
	\end{itemize}
Our main result reads:
\begin{theorem}
	\label{thm:main}
	Assume \eqref{eq:G estimate}, \eqref{eq:G delta Linf to delta C}, \eqref{eq:S submarkovian} and \eqref{eq:Green normal}. Given $u_0\in L^1(\Omega,\delta^\gamma)$, $f\in L^1(0,T;L^1(\Omega,\delta^\gamma))$ and $h\in L^1((0,T)\times\partial\Omega)$, the Problem \eqref{eq:main}-\eqref{eq:sing.bc}-\eqref{eq:sing.ext}-\eqref{eq:init} has a unique solution $\Heat[u_0,f,h]\in L^1(0,T;L^1(\Omega,\delta^\gamma))$ given in \eqref{eq:general form of the solution}, understood in the weak-dual sense \eqref{eq:weak phi}. Moreover, the solution operator is continuous in the sense that \eqref{eq:L1 estimate} holds.
\end{theorem}

\noindent A word about notation. We denote by  $L^1(\Omega,\delta^\gamma)$ the space of functions $u\in L^1_{loc}(\Omega)$ such that
 $$
 \int_{\Omega}|u(x)|\delta(x)^\gamma \diff x <\infty.
 $$
Similarly, we define
\[
L^1(0,T;L^1(\Omega,\delta^\gamma))
=\set{
    u\in L^1_{\rm loc}((0,T)\times \Omega):
    \int_0^T\int_{\Omega}
        |u(x)|\,\delta(x)^\gamma
    \diff x\diff t<\infty}.
\]
After recalling some elliptic preliminaries, in \Cref{sec:semigroup} we will use the theory of dissipative operators to construct the heat semigroup $\Semi(t)$ defined by $\Semi(t)[u_0] = \Heat[u_0,0,0]$. In \Cref{sec:introduction to singular boundary data} we give some intuition on why and how  singular boundary data are allowed into the theory.
In \Cref{sec:weyl} we provide a one-sided Weyl-type law for such operators, which is used in \Cref{sec:heat kernel existence and estimates}
to show that the heat semigroup is regularising, it is in fact given by a kernel
with good estimates.

\Cref{sec:general form} is devoted to propose an explicit candidate of solution  $\Heat[u_0,f,h]$ for Problem \eqref{eq:main}--\eqref{eq:init}, that is given in terms of the heat kernel $\p$ defined in \eqref{eq:Semi given kernel}. We analyse the admissible data, and obtain basic  estimates. 
	In \Cref{sec:weak formulation} we give a precise definition of weak-dual solution, for which we show uniqueness. We devote \Cref{sec:existence} to show that our candidate of solution is precisely this unique solution and that \Cref{thm:main} holds.
	
	In \Cref{sec:elliptic and parabolic} we address the question of agreement between the elliptic and parabolic theories. Thus, we prove that when $f$ and $h$ do not depend on $t$, then $\Heat[u_0,f,h]$ converges as $t \to \infty$ to the solution of the corresponding elliptic problem studied in \cite{abatangelo+gc+vazquez2019}.
	
To conclude, we examine in \Cref{sec:examples} the practical application of the above general theory. In particular, we show that the theory applies to the three classical fractional Laplacian examples: RFL, SFL and CFL, by checking that these operators satisfy the  set of hypotheses required by our general setting.

\section{Elliptic preliminaries}\label{sec.ellprl}
We recall some facts from the elliptic theory as developed by Abatangelo et al. in \cite{abatangelo+gc+vazquez2019}. First, we go over the theory where only zero boundary data are taken into account (we often refer to it as standard theory). Thus, it is proved that for a large family of operators $\Ls$, the classical solution of
	\begin{equation*}
	\begin{dcases}
	\Ls v= f,  & \text{ in }\Omega, \\
	v = 0, & \text{ in }\Omega^c,
	\end{dcases}
	\end{equation*}
	can be written by an integral against the Green kernel	
\begin{equation}\label{greenG}
	\Green[f] (x) \defeq v(x) = \int_\Omega \G (x,y) f(y) \diff y.
\end{equation}
This kernel has  very precise properties.
	In the class of operators under consideration, for all $0 \le f \in L^\infty_c (\Omega)$ the solutions behave near the boundary in a power-like way
	\begin{equation*}
	v \asymp \delta^\gamma, \qquad    \delta(x)=\mbox{dist}(x,\partial \Omega),
	\end{equation*}
with a certain constant $\gamma \in (0,1]$ that depends on the operator. See the classical fractional Laplacian examples in Section \ref{sec:examples}.	We assume that $\G$ is symmetric and has the following estimates outside the diagonal,
	\begin{align}\label{eq:G estimate}	\tag{H$_1$}
		\G(x,y) = \G(y,x) \asymp |x-y|^{-(n-2s)} \left( \frac{\delta(x) \delta(y)}{|x-y|^2} \wedge 1 \right)^{\gamma}.
	\end{align}
	Under this assumption, it was shown in \cite{abatangelo+gc+vazquez2019} that the maximal class of data $f$ admissible for the operator is the space $L^1 (\Omega, \delta^\gamma)$, which only depends on the parameter $\gamma$.
	
On the other hand, Bonforte et al.\  \cite{bonforte+figalli+vazquez2018} showed that there exists an orthonormal {basis of $L^2 (\Omega)$ consisting of eigenfunctions $\varphi_m$} of $\Ls$. We write this in terms of the inverse $\Green[\varphi_m] = \lambda_m^{-1} \varphi_m$ 	where the sequence of eigenvalues $0 < \lambda_1 < \lambda_2 \le \cdots$ diverges to infinity in the usual way. In \Cref{sec:weyl} we include a Weyl-type law that specifies a rate for this divergence.

Let us use the notation $\delta^\gamma X \defeq \{ u=\delta^\gamma v :  v \in X  \}$ endowed with the norm $\norm[\delta^\gamma X]{u}=\norm[X]{u/\delta^\gamma}$.  With the additional assumption that		\begin{equation}	\tag{H$_2$}
	\label{eq:G delta Linf to delta C}
	\Green: \delta^\gamma L^\infty (\Omega) \to \delta^\gamma C(\overline \Omega) \text{ is continuous},
	\end{equation}
	it is known that the eigenfunctions are bounded functions such that
	\begin{equation}
	\label{eq:regularity of the eigenf}
		\varphi_1 \asymp \delta^\gamma, \qquad \varphi_k \in \delta^\gamma C (\overline \Omega).
	\end{equation}	
	Through this eigen-decomposition, we have the so-called  Mercer's condition on the kernel
	\begin{equation}
		\label{eq:Mercer}
		\int_\Omega \int_\Omega \G(x,y) f(x) f(y)  \diff y \diff x = \int_\Omega \Green[f]f = \sum_{k=1}^\infty
		\lambda_k^{-1}
		\angles{f,\varphi_k}^2 \ge 0 , \qquad \forall f \in L^2(\Omega).
	\end{equation}
	Here and below, we denote by $\langle \cdot, \cdot \rangle$ the scalar product on $L^2 (\Omega)$,
	\begin{equation*}
		\angles{u,v} = \int_\Omega u(x) v(x) \diff x.
	\end{equation*}

\section{Semigroup theory}
\label{sec:semigroup}
 We start our evolution study by the simplest  theory when both $f=0$ and $h = 0$, and only nontrivial initial data $u_0$ are considered. Then we present a natural class solutions that forms a continuous semigroup that creates the basis for our further studies. We  will denote it as	$\Semi(t) [u_0]$, i.e.,
\begin{align*}
	\Semi(t) [u_0] (x) \defeq \Heat_\Ls[u_0,0,0] (t,x) .
	\end{align*}
We  construct this heat semigroup $\Semi(t) : L^2 (\Omega) \to L^2 (\Omega)$ from the eigen-decomposition
	\begin{equation}
	\label{eq:heat semigroup}
		\Semi(t) [u_0] = \sum_{k=1}^\infty e^{-\lambda_k t} \langle u_0 , \varphi_k \rangle \varphi_k.
	\end{equation}
This is the classical form of the of solution of \eqref{eq:main}--\eqref{eq:init}  when $u_0 \in L^2 (\Omega), f = 0, h = 0$. It satisfies the equation in the spectral sense. A definition of solution of the general problem will be given later; we are satisfied with this formal definition for the moment. Agreement between the present semigroup definition and the general theory to be developed later will come at the proper place.
	
	\begin{remark}
		\label{rem:S self-adjoint}
		There are some properties that follow from \eqref{eq:heat semigroup}. First, $\Semi(t)$ is a well-defined self-adjoint linear operator in $L^2(\Omega)$, namely
		\begin{equation*}
		\int_\Omega \Semi(t) [f] g
 = \sum_{k=1}^\infty  e^{-\lambda_k t} \langle f, \varphi_k \rangle \langle g, \varphi_k \rangle
= \int_\Omega f \Semi(t)[g].
		\end{equation*}
		Second,  it easily follows that for all $f \in L^2 (\Omega)$,
		\begin{equation}
		\label{eq:G from S}
		\Green [f] =
		\sum_{k=1}^\infty \lambda_k^{-1} \langle f, \varphi_k \rangle \varphi_k =
		 \int_0^\infty \Semi(t) [f] \diff t.
		\end{equation}
		In the following subsections we check that $\Semi(t)$ is, in fact, a semigroup in various functional settings.
	\end{remark}

\subsection{\texorpdfstring{$L^2$}{L2} theory}
	
	\begin{proposition}
		The family $\Semi(t)$ defined by \eqref{eq:heat semigroup} is a continuous  non-expansive  semigroup in $L^2 (\Omega)$. Furthermore,
		\begin{equation*}
			\| \Semi(t) [u_0] \|_{L^2 (\Omega)} \le e^{-\lambda_1 t} \|u_0 \|_{L^2 (\Omega)}.
		\end{equation*}
	\end{proposition}
	\begin{proof}
		First, it is evident that
		\begin{equation*}
		\Semi(t) \Semi(\tau) = \Semi(t+ \tau),
		{\qquad \forall t,\tau\geq 0.}
		\end{equation*}
		Now we show that $\Semi(t) \to I$ strongly in $L^2(\Omega)$:
		\begin{equation*}
		\left \|  \Semi(t)[u_0 ] - u_0 \right \|_{L^2 (\Omega)}^2  = \sum_{k=1}^\infty \left( 1 - e^{-\lambda_k t} \right)^2
		\angles{ u_0 , \varphi_k }^2
		 \le \left( 1 - e^{-\lambda_1 t} \right)^2 \| u_0 \|_{L^2 (\Omega)} .
		\end{equation*}
		Since $e^{-\lambda_1 t} \to 1 $, we have the continuity of the semigroup.
		The estimate is computed similarly.
	\end{proof}

\begin{remark}
	The Green kernel
	can be formally expressed as
	\begin{equation*}
	\G(x,y) =  \sum_{{k=1}}^{\infty} \lambda_k^{-1} \varphi_k (x) \varphi_k (y).
	\end{equation*}
	We expect also that
	\begin{equation}
	\label{eq:Semi given kernel}
	\Semi(t) [u_0] (t,x)
	= \int_{\Omega} \p(t,x,y) u_0 (y) \diff y,
	\end{equation}
	where $\p$ is called the heat kernel, whose existence is justified in \Cref{thm:S boundary reg}. Formally, we can write
	\begin{equation}
	\label{eq:S from G}
	\p(t,x,y) = \sum_{{k=1}}^{\infty} e^{-\lambda_k t} \varphi_k (x) \varphi_k (y).
	\end{equation}
	Since $\Semi(t)$ is self-adjoint, $\p(t,x,y) = \p(t,y,x)$. Through these formulas, it is immediate that
	\begin{equation*}
	\G (x,y) = \int_0^\infty \p(t,x,y) \diff t.
	\end{equation*}
This relation will be justified in \Cref{rmk:rigorous parabolic}.
	\end{remark}

	\begin{remark}
		\label{rem:energy setting}
		As mentioned in \cite{chan+gc+vazquez2020}, in this $L^2$ setting we can define the energy spaces
		\begin{equation*}
			H_\Ls^1 (\Omega) = \set{   u \in L^2 (\Omega) : \sum_{k=1}^\infty \lambda_k \langle u, \varphi_k \rangle ^2 <\infty}, \quad 			H_\Ls^2 (\Omega) = \set{   u \in L^2 (\Omega) : \sum_{k=1}^\infty \lambda_k^2 \langle u, \varphi_k \rangle ^2<\infty} = \Green [L^2 (\Omega)].
		\end{equation*}
		It is easy to see that these are Hilbert spaces with adequate inner products. In the examples, the spaces have been characterised. The dual spaces are made of distributions
		and expressed as
		\begin{equation*}
		H_\Ls^{-1} (\Omega) = \set{   u \in D'(\Omega)  : \sum_{k=1}^\infty \lambda_k^{-1} \langle u, \varphi_k \rangle ^2<\infty},
		\qquad 			
		H_\Ls^{-2} (\Omega) = \set{   u \in D'(\Omega) : \sum_{k=1}^\infty \lambda_k^{-2} \langle u, \varphi_k \rangle ^2<\infty} .
		\end{equation*}
		The heat semigroup $\Semi(t)$ is still well defined in this setting.
	\end{remark}

\subsection{\texorpdfstring{$L^\infty$}{L infty} theory}\label{ssec.32}
	The usual assumption on the operator $\Ls$ is that it is submarkovian, namely
	\begin{equation}
	\label{eq:S submarkovian}
	\tag{H$_3$}
	0\le u_0 \le 1 \implies 0 \le \Semi(t) [u_0]  \le 1 , \qquad \forall  t \geq 0 .
	\end{equation}
Notice that, since
\begin{equation*}
0\le \frac{u_0 (x) - \essinf u_0}{\esssup u_0 - \essinf u_0 } \le 1, \quad \text{ a.e. } x \in \Omega  \qedhere
\end{equation*}
we have that \eqref{eq:S submarkovian} is equivalent to
\begin{equation}
\essinf u_0 \le \Semi(t) [u_0] \le \esssup u_0, \qquad \forall t \ge 0.
\end{equation}
If these properties hold, $\Semi(t)$ is positivity-preserving and non-expansive on $L^\infty(\Omega)$, namely
\begin{equation}\label{eq:max prin}
u_0 \geq 0
\quad \Longrightarrow \quad
\Semi(t)[u_0]\geq 0,
\qquad \forall t\geq 0,
\end{equation}
\begin{equation}\label{eq:L infty nonexpansive}
|u_0| \leq 1
\quad \Longrightarrow \quad
\bigl|\Semi(t)[u_0]\bigr| \leq 1
\qquad \forall t\geq 0.
\end{equation}
In particular, this assumption guarantees that $\p\geq0$ (see \Cref{thm:S boundary reg})
and a more useful version of \eqref{eq:L infty nonexpansive}:
\begin{equation}
\label{eq:abs value comparison}
\left| \Semi(t) [u_0] \right| \le \Semi(t) [ |u_0| ].
\end{equation}
This estimate will be fundamental in our theory below.
The symmetry assumption \eqref{eq:G estimate} implies that $\Semi(t)$ is a symmetric Markov semigroup.

\subsection{\texorpdfstring{$L^1$}{L1} theory}
The sub-markovian condition \eqref{eq:S submarkovian} also implies $\Semi(t):L^1(\Omega)\longrightarrow L^1(\Omega)$ is continuous and non-expansive. Indeed, using also \eqref{eq:abs value comparison} and the self-adjointness of  $\Semi(t)$, we have for any $u_0\in L^1(\Omega)$,
\begin{equation}
\label{eq:S L1 cont}
\int_{\Omega}
    |\Semi(t)[u_0](x)|
\diff x
\leq
\int_{\Omega}
    \Semi(t)[|u_0|](x)
\diff x
=\int_{\Omega}
    |u_0(x)|\Semi(t)[1](x)
\diff x
\leq
\int_{\Omega}
    |u_0|
\diff x.
\end{equation}
Just as in the classical setting, such duality estimates will be useful in the subsequent theory.

\subsection{\texorpdfstring{$L^p$}{Lp} theory}

    Since $\Semi(t)$ is non-expansive on $L^\infty(\Omega)$ and $L^1(\Omega)$, it is immediate by Riesz--Thorin interpolation theorem that $\Semi(t):L^p(\Omega) \longrightarrow  L^p(\Omega)$ defines a continuous and non-expansive semigroup for $1<p<\infty$. Moreover, since we know later from \Cref{thm:S boundary reg} that $\Semi(t)$ is ultracontractive,
    i.e. $\Semi(t):L^2(\Omega) \longrightarrow  L^\infty(\Omega)$ is continuous for any $t>0$.
    Due to \cite[Theorem 2.1.5]{Davies1989}, $\Semi(t)$ is compact on $L^p(\Omega)$ for all $1\leq p \leq \infty$. However, we will not use this fact in the rest of the paper.

	\subsection{\texorpdfstring{$L^1 (\Omega, \delta^\gamma)$}{L1(Omega,delta gamma)} theory}

   	In \cite{abatangelo+gc+vazquez2019}, the authors prove that,  from $\G(x,y) \ge c \delta(x)^\gamma \delta(y)^\gamma$,
   	we have the lower-Hopf-type inequality for $f \ge 0$,
   	\begin{equation*}
   	\Green[f] (x) \ge c \delta(x)^\gamma \int_\Omega f(y) \delta(y)^\gamma \diff y.
   	\end{equation*}
   	This allows them to conclude that $L^1 (\Omega, \delta^\gamma)$ is the optimal set of functional data. Since the operator $\Semi(t)$ is self-adjoint in $L^2(\Omega)$ (see \Cref{rem:S self-adjoint}),
   we have the weaker information that
   	\begin{equation*}
   	\int_\Omega \Semi(t) [u_0] \varphi_1 \diff x = e^{-\lambda_1 t} \int_\Omega u_0 \varphi_1 \diff x,
        \qquad \forall u_0\in L^2(\Omega).
   	\end{equation*}
   	Therefore, for $u_0 \ge 0$ we have that
   	 $u_0 \in L^1 (\Omega, \delta^\gamma)$ if and only if
   	 $\Semi(t)[u_0] \in  L^1 (\Omega, \delta^\gamma)$ for any $t > 0$.
   	In \Cref{rem:Hopf for t large} we show that if $u_0 \notin L^1 (\Omega, \delta^\gamma)$ then $\Semi(t) [u_0] \equiv +\infty$ for $t$ large enough. While we do not know if the same happens in general for small times, we show in \Cref{sec:examples} that for the model operators
    this holds for all $t>0$. Therefore, the sensible set of optimal functional data is, as in the elliptic case, $u_0 \in L^1 (\Omega, \delta^\gamma)$.
   		
   	Since we want to work in the  setting of semigroups associated to an infinitesimal generator (see, e.g., \cite{Pazy1983}) we %
   formalise the functional setting of our infinitesimal generator, and %
   apply the well-known Hille--Yosida theorem (see, e.g. \cite[Chapter 1, Theorem 3.1]{Pazy1983}), as we recall as follows:
   \begin{theorem*}[Hille--Yosida]
   A linear (unbounded) operator $A$ in a Banach space $X$ is the infinitesimal generator of a $C_0$ semigroup of contractions %
   if and only if
   \begin{enumerate}
   \item [(i)] $A$ is closed and $\overline{D(A)}=X$.
   \item [(ii)] The resolvent set of $A$ contains $(0,+\infty)$ and for every $\lambda>0$, the resolvent operator $\J_\lambda=(\lambda I-A)^{-1}$ has operator norm bounded by $       \norm[]{\J_\lambda}
       \leq 1/\lambda.$
   \end{enumerate}
   \end{theorem*}
   	
   	The infinitesimal operator includes somehow the boundary or exterior conditions, and we will call it $-A$ rather that $\Ls$ to avoid confusion so that problem is the ODE
   	\begin{equation*}
   		u' = A u.
   	\end{equation*}
   	Since $\Green$ is injective, we can define
   	\begin{equation}
   	D(A) = \Green (L^1 (\Omega, \delta^\gamma)) , \qquad Au \defeq - \Green^{-1} [u],\ \forall u \in D(A).
   	\end{equation}
   	This is to say that $Au$ is the unique element $f \in L^1 (\Omega, \delta^\gamma)$ such that $u = -\Green[f]$. This is operator is linear and possibly unbounded. Due to the characterisation given by the weak-dual formulation
   	\begin{equation*}
   	\langle - Au , \Green [\psi] \rangle = \langle u , \psi \rangle, \qquad \forall \psi \in L_c^\infty (\Omega).
   	\end{equation*}
   	we know that this operator is closed, i.e. if $u_n \to u$ and $Au_n \to v$ in $L^1 (\Omega, \delta^\gamma)$, then $v = A u$. Since $\varphi_k \in D(A)$, we know that $\overline{D(A)} = L^1 (\Omega, \delta^\gamma)$.
   	
   	\begin{remark}
   		Notice that if $u \in \Green (L^2 (\Omega))$ then
   		\begin{equation*}
   			Au = - \sum_{k=1}^\infty \lambda_k  \langle u , \varphi_k \rangle \varphi_k.
   		\end{equation*}
   	\end{remark}
   	
   	\medskip
   	
   	The second hypothesis
   of the Hille--Yosida theorem is that for every $\lambda > 0$ the resolvent $\J_\lambda = ( \lambda I - A)^{-1}$ is a bounded linear operator such that
   	\begin{equation*}
   		\| \J_\lambda [f] \|  \le \frac 1 \lambda \| f \|.
   	\end{equation*}
   	for some norm of $L^1 (\Omega, \delta^\gamma)$.
   	Notice that
   	\begin{equation*}
   	\J_\lambda[f] \text{ is the unique solution of }   u +  \lambda \Green[u] = \Green[f].
   	\end{equation*}
   	Due to \eqref{eq:Mercer}, following \cite[Theorem 4.1 and Theorem 4.2]{GC+Vazquez2018}, there exists a unique solution and $ |u| \le \Green (|f|)$. Therefore, there the resolvent is well defined $\J_\lambda : L^1 (\Omega, \delta^\gamma) \to L^1 (\Omega, \delta^\gamma)$. Furthermore, splitting into positive and negative parts and testing each against the first eigenfunction we have that
   	\begin{equation*}
   	\int_\Omega |\J_\lambda [f]| \varphi_1 \le \frac{1}{\lambda + \lambda_1} \int_ \Omega |f| \varphi_1.
   	\end{equation*}
   	Since $\lambda_1 > 0$, the resolvent is contracting in the equivalent norm of $L^1 (\Omega ,\delta^\gamma)$ given by $\| u \| = \int_\Omega |u| \varphi_1$.
   	\normalcolor

    \begin{proposition}
    \label{prop:semigp-L1-def}
    	The family $\Semi(t)$ defined as the unique extension of \eqref{eq:heat semigroup} is a $C_0$ semigroup in $L^1 (\Omega, \delta^\gamma)$.
    \end{proposition}
    \begin{proof}
    Through the Hille--Yosida theorem, there exists a $C_0$ semigroup $\Semi_1(t)$ associated to $A$. For $u_0 =  \sum_{m=1}^M c_k \varphi_k$ the solution is classical and defined by
	\begin{equation*}
		\Semi_1 (t) [u_0] = \sum_{m=1}^M e^{-\lambda_k t} c_k \varphi_k = \Semi(t) [u_0]
	\end{equation*}
	Since the set $\{\sum_{m=1}^M c_k \varphi_k : M \in \mathbb N, c_k \in \mathbb R\}$ is dense in $L^2 (\Omega)$, we have $\Semi_1 (t) = \Semi (t)$ on $L^2 (\Omega)$. This completes the proof.
	\end{proof}
	We conclude this section with a proof of the decay of the weighted norm.
	\begin{proposition}
		\label{prop:decay of L1 delta}
		Let $ u_0 \in L^1 (\Omega, \delta^\gamma)$. Then
		\begin{equation*}
    	 \| \Semi(t) [u_0] \delta^\gamma \|_{L^1 (\Omega)}
    \le C e^{-\lambda_1 t}
        \| u_0 \delta^\gamma  \|_{L^1 (\Omega)}. 
		\end{equation*}
	\end{proposition}
	\begin{proof}
		Let $0 \le u_0 \in L^2 (\Omega)$. Then
		\begin{equation*}
			\int_\Omega \Semi(t) [u_0] \varphi_1 = e^{-\lambda_1 t} \int_\Omega u_0 \varphi_1.
		\end{equation*}
		If $u_0 \in L^2 (\Omega)$ changes sign, we apply \eqref{eq:abs value comparison} to show that
		\begin{equation*}
			\int_\Omega | \Semi(t) [u_0] |  \varphi_1
\le \int_\Omega \Semi(t) [|u_0|] \varphi_1
 = e^{-\lambda_1 t} \int_\Omega |u_0| \varphi_1.
		\end{equation*}
		Since $\varphi_1 \asymp \delta^\gamma$, the result is proven in this case.
		Any $u_0 \in L^1 (\Omega, \delta^\gamma)$ can be approximated by a sequence $u_{0,k} \in L^2 (\Omega)$. Since $\Semi(t)$ is continuous in $L^1 (\Omega, \delta^\gamma)$, we conclude the result.
	\end{proof}
	\normalcolor

	\subsection{Duhamel's formula}
	 In the setting of $u_0=0$, $f\neq 0$, $h=0$,  we can use the classical formula by Duhamel to solve the problem with a forcing term:
	\begin{equation}
	\label{eq:Duhamel}
		\Heat[0,f,0] (t,x) = \int_0^t \Semi(t-\sigma) [f(\sigma, \cdot)] (x) \diff \sigma,
	\end{equation}
	 or more explicitly,
	\begin{equation*}
			\Heat[0,f,0] (t,x) = \int_0^t \int_\Omega \p (t-\sigma, x,y) f(\sigma, y) \diff y \diff \sigma.
	\end{equation*}
    Due to the admissible data for the semigroup it is natural to request that $f(t, \cdot) \in L^1 (\Omega, \delta^\gamma)$. It suffices that $f \in L^1(0,T; L^1 (\Omega, \delta^\gamma))$. We provide now some immediate estimates, %
    \[\begin{split}
    |\Heat[0,f,0](t,x)|
    &\leq \Heat[0,|f|,0](t,x)\\
    \int_{\Omega}
        |\Heat[0,f,0](t,x)|\varphi_1(x)
    \diff x
    &\leq
    \int_{0}^{t}
        e^{-\lambda_1(t-\sigma)}
        \int_{\Omega}
            |f(\sigma,y)|\varphi_1(y)
        \diff y
    \diff \sigma.
    \end{split} \]
    Hence,
    \begin{align}
    \int_{\Omega}
        |\Heat[0,f,0](t,x)|\delta(x)^\gamma
    \diff x
    &\leq C
    \int_{0}^{t}
        \int_{\Omega}
            |f(\sigma,y)|\delta(y)^\gamma
        \diff y
    \diff \sigma.
    \end{align}
	We will develop a more detailed theory in \Cref{sec:general form}. Note that no  (singular)  boundary data are considered in this section.

\section{Introduction to singular boundary data}
\label{sec:introduction to singular boundary data}
We are  ready to address the evolution problem involving nontrivial boundary data.

\subsection{Review of the elliptic theory}
	A detailed analysis of the kernel done in \cite{abatangelo+gc+vazquez2019} showed that $\Green$ is defined in $L^1 (\Omega, \delta^\gamma)$ but cannot be extended to a larger set of functions in a reasonable way.
	Due to \eqref{eq:G delta Linf to delta C}, we know that if $\psi \in \delta^\gamma L^\infty (\Omega)$ and $\varphi = \Green[\psi]$, then the $\gamma$-normal derivative
	\begin{equation*}
		\D_\gamma \varphi (\zeta) \defeq
		 \lim_{\substack{x \to \zeta\\x\in\Omega}} \frac{\varphi(x)}{\delta(x)^\gamma}
	\end{equation*}
 exists for each $\zeta \in \partial \Omega$ and can be taken uniformly in $\zeta$.
	In \cite{abatangelo+gc+vazquez2019}, under the additional assumption that for every $y \in \Omega$ and $\zeta \in \Omega$ there exists a limit
	\begin{equation}
	\label{eq:Green normal}
	\tag{H$_4$}
	\D_\gamma \G (\zeta,y) \defeq \lim_{x \to \zeta} \frac{\G(x,y)}{\delta(x)^\gamma}
	\qquad \forall \zeta \in \partial \Omega,
	\, y\in\Omega,
	\end{equation}
	\normalcolor 	
	the authors showed that the problem
	\begin{equation*}
		\begin{dcases}
		\Ls v= 0,  & \text{ in } \Omega, \\
		\E v = h, & \text{ on } \partial \Omega, \\
		v = 0, & \text{ in } \Omega^c,
		\end{dcases}
	\end{equation*}
	has a unique solution in the sense that
	\normalcolor
	\begin{equation}
	\label{eq:Martin weak formulation}
	\int_\Omega v  \psi  = \int_{\partial \Omega} h \D_\gamma \Green [\psi], \qquad \forall \psi \in \delta^\gamma L^\infty (\Omega),
	\end{equation}
	which we will denote by $v = \Martin[h]$.  Since $\D_\gamma\Green[\psi]\in L^\infty(\partial\Omega)$ due  to \eqref{eq:G delta Linf to delta C}, this equation is well posed if $h \in L^1 (\partial \Omega)$. The function $u^\star$ appearing in \eqref{eq:E definition} is precisely
	\begin{equation*}
		u^\star \defeq \Martin [1].
	\end{equation*}
	If in addition  $h \in C(\partial \Omega)$, the authors prove that $\E [v] = h$ is satisfied in the pointwise sense.
	These
	\normalcolor
	solutions can be obtained from the ``usual'' elliptic problem, by considering a sequence
		\begin{equation*}
	\begin{dcases}
	\Ls v_m= f_m,  & \text{ in } \Omega, \\
	v_m = 0, & \text{ in }\Omega^c,
	\end{dcases}
	\end{equation*}
	for $f_m$ concentrating towards the boundary in the form
	\begin{equation*}
		f_m (x) = \frac{|\partial \Omega|}{|A_m|} \frac {\chi_{A_m} (x)}{\delta(x)^\gamma} h(P_{\partial \Omega } (x)),
	\end{equation*}
	where $A_m$ is the set of points at distance {between $1/m$ and $2/m$}
from $\partial \Omega$, and $P_{\partial \Omega}$ is the orthogonal projection on $\partial \Omega$ given by the tubular neighbourhood problem. The idea is to prove uniform integrability of the sequence $v_m = \Green[f_m]$ against test functions, and pass to the limit in the weak-dual formulation \begin{equation*}
	\int_\Omega v_m \psi = \int_\Omega f_m \Green[\psi] = \frac{|\partial \Omega|}{|A_m|} \int_{A_m} h(P_{\partial \Omega } (x)) \frac {\Green[\psi](x)}{\delta(x)^\gamma} \diff x , \qquad \forall \psi \in \delta^\gamma L^\infty (\Omega).
\end{equation*}
As $m \to +\infty$ we have formally that $\Ls [v]  = 0$, and rigorously that
\begin{equation*}
		\int_\Omega v \psi = \int_{\partial \Omega} h(\zeta) \D_\gamma  {\Green[\psi](\zeta)} \diff \zeta  , \qquad \forall \psi \in \delta^\gamma L^\infty (\Omega).
\end{equation*}
\normalcolor
Letting $\varphi = \Green[\psi]$ this weak formulation is equivalent to the existence of an integration-by-parts formula  satisfied by functions with zero exterior condition (if applicable),
\begin{equation}
\label{eq:integration by parts}
\int_\Omega v  \Ls[ \varphi]   = \int_\Omega \Ls [v] \varphi + \int_{\partial \Omega} \E [v]  \D_\gamma [\varphi].
\end{equation}
This kind of integration-by-parts formula was known for model operators in \cite{Abatangelo2015,Abatangelo2015b,Abatangelo2017a}.

\medskip

	Passing to the limit in \eqref{eq:G estimate}, we have the estimate
	\begin{equation}
    \label{eq:D gamma G estimate}
	\D_\gamma \G (\zeta,y) \asymp |\zeta -y|^{-(n-2s+2\gamma)} \delta(y)^\gamma.
	\end{equation}
	Then, for any $\psi \in L^\infty_c (\Omega)$ there exists $\D_\gamma [\Green [\psi]]$ given by
	\begin{equation*}
	\D_\gamma [\Green[\psi]] (\zeta) = \int_{\Omega} \D_\gamma \G (\zeta, y) \psi(y) \diff y.
	\end{equation*}
	From here on, we simply denote this by $\D_\gamma \Green [\psi]$.
	Applying Fubini's theorem in \eqref{eq:Martin weak formulation} we deduce that
	\begin{equation*}
	\Martin[h]{(x)} = \int_{\partial \Omega} \M (x,\zeta) h(\zeta) \diff \zeta, \qquad \text{ where } \M (x, \zeta) = \D_\gamma \G (\zeta, x).
	\end{equation*}
	Again, this kind of representation was known for some of the examples, see \cite{Abatangelo2015,Abatangelo2015b,Abatangelo2017a}. Note, in particular, that from \eqref{eq:D gamma G estimate} and the homogeneity $n-1$ of the surface measure on $\partial\Omega$, it is not difficult to recover that
	\begin{equation*}
	u^\star \asymp \delta^{2s-\gamma-1}.
	\end{equation*}
	Notice that the exponents involved in $\D_\gamma \G$ and $u^\star$ are related by the numerical relation:
\[
-(n-2s+2\gamma)+\gamma+(n-1)=2s-\gamma-1.
\]

	\subsection{Intuition for the parabolic problem}
    \label{sec:parabolic intuition}

	 As long as the limit
	\begin{equation}
	\label{eq:D gamma S defn}
	{\D_\gamma} \p (t,\zeta,y) \defeq \lim_{ x \to \zeta} \frac{\p(t,x,y) }{\delta(x)^\gamma}
	\end{equation}
 exists uniformly in $\delta^\gamma C(\overline\Omega)$, for any weighted measure data
 \begin{equation*}
 	u_0\in M(\Omega,\delta^\gamma) = \left \{ \mu \text{ measure} :  \int_{\Omega}\delta^\gamma\,d|\mu| < \infty \right \}
 \end{equation*}
	we can write
	\begin{align*}
	\D_\gamma\Semi(t)[u_0](\zeta)
    &= \D_\gamma \Heat [u_0, 0,0 ]  (t,\zeta)
   		= \int_{\Omega} \D_\gamma \p(t,\zeta,y) u_0 (y) \diff y.
	\end{align*}
	Due to \eqref{eq:S submarkovian}, we know that  $\p \ge 0$ and hence $\D_\gamma \p \ge 0$.
	Therefore
	\begin{align}
	\label{eq:D gamma H 0 f 0}
		\D_\gamma \Heat [0,f,0 ] (t,\zeta) &= \int_0^t  \int_{\Omega} \D_\gamma \p(t-\sigma ,\zeta,y) f ( \sigma,y) \diff y \diff \sigma.
	\end{align}
	\begin{remark}
		 In view of  \eqref{eq:S from G} we can formally write
		\begin{equation*}
			\D_\gamma \p(t, \zeta,y) =  \sum_{{m=1}}^{\infty} e^{-\lambda_m t} \D_\gamma \varphi_m (\zeta) \varphi_m (y).
		\end{equation*}
The convergence of this series can be justified by the polynomial growth on the eigenfunctions together with their $\gamma$-normal derivatives (see \eqref{eq:varphi-bound}), as well as the one-sided Weyl's law presented in \Cref{thm:weyl}. 
	\end{remark}
	
	We can repeat this general scheme, and construct
a boundary singular solution when $0\neq h\in C(\partial\Omega)$. Moreover, we see that a boundary Duhamel-type formula holds, see \eqref{eq:Duhamel-boundary}. For more general boundary data, we will use the weak-dual formulation. Define the sequence of compactly supported functions that concentrates towards the boundary data,
	\begin{equation}
	\label{eq:concentration to boundary f_j}
		f_j (t,x) = \frac{|\partial \Omega|}{|A_j|} \frac {\chi_{A_j} (x)}{\delta(x)^\gamma}h(t,P_{\partial \Omega } (x)).
	\end{equation}
	Applying Duhamel's formula for $f_j$ we have
	\begin{equation}
	\label{eq:towards Martin passage to the limit}
	\begin{split}
		u_j (t,x)
        &= \int_0^t \Semi(t-\sigma) [f_j(\sigma, \cdot)] (x) \diff \sigma\\
		&=\int_0^t \int_\Omega  \p(t-\sigma,x,y) f_j(\sigma, y) \diff y \diff \sigma\\
		&=\int_0^t  \frac{|\partial \Omega|}{|A_j|} \int_{A_j}  \p(t-\sigma,x,y) \frac {h(\sigma,P_{\partial \Omega } (y))}{\delta(y)^\gamma}  \diff y \diff \sigma.
	\end{split}
	\end{equation}
	If the right hand side has a limit as $j\to\infty$, we venture to speculate that it is precisely $\Heat[0,0,h]$, i.e.
	\normalcolor
	\begin{equation}\label{eq:Duhamel-boundary}
		\Heat[0,0,h] (t,x) \defeq \int_0^t \int_{\partial \Omega} \D_\gamma \p (t - \sigma,\zeta,x) h(\sigma, \zeta)\diff \zeta \diff \sigma.
	\end{equation}

    For time-independent boundary data, we have

\begin{remark}
	\label{rem:satisfies boundary condition}
Let $h(t,x)\equiv h(x)$ be a given boundary datum. We point out several facts which we write in formal terms:
\begin{enumerate}
	\item
The stationary function $\Martin[h](x)$ solves the initial-boundary problem with $u_0 = \Martin[h](x)$ and boundary data $h$, i.e.
\begin{equation}\label{eq:singular data equality}
\Heat\Big [\Martin[h],0,h \Big ](t,x) = \Martin[h](x)
    \qquad \forall x\in\Omega,\,t>0.
\end{equation}
Indeed, using formally the representation  of the kernels
we have
\[\begin{split}
\Martin[h]
&=\sum_{m=1}^{\infty}
\dfrac{1}{\lambda_m}
\left(
\int_{\partial\Omega}
\D_\gamma\varphi_m(\zeta)h(\zeta)
\diff \zeta
\right)
\varphi_m(x).
\end{split}\]
whereas
\[\begin{split}
\Heat[0,0,h](t,x)
&=\int_{0}^{t}
    \int_{\partial\Omega}
    \sum_{m=1}^{\infty}
        e^{-\lambda_m(t-\sigma)}
        \D_\gamma\varphi_m(\zeta)
        \varphi_m(x)
    h(\zeta)
\diff \zeta
\diff \sigma\\
&=\sum_{m=1}^{\infty}
    \dfrac{1-e^{-\lambda_m t}}{\lambda_m}
    \left(
        \int_{\partial\Omega}
            \D_\gamma\varphi_m(\zeta)h(\zeta)
        \diff \zeta
    \right) \varphi_m (x)\\
\Heat[\Martin[h],0,0](t,x)
&=
\int_{\Omega}
\sum_{m=1}^{\infty}
    e^{-\lambda_m t}
    \varphi_m(x)
    \varphi_m(y)
\sum_{\ell=1}^{\infty}
    \dfrac{1}{\lambda_\ell}
    \varphi_\ell(x)
    \left(
        \int_{\partial\Omega}
            \D_\gamma\varphi_\ell(\zeta)h(\zeta)
        \diff \zeta
    \right)
\diff x\\
&=\sum_{m=1}^{\infty}
    \dfrac{e^{-\lambda_m t}}{\lambda_m}
    \left(
        \int_{\partial\Omega}
            \D_\gamma\varphi_m(\zeta)h(\zeta)
        \diff \zeta
    \right)
    \varphi_m(x).
\end{split}\]
The equality \eqref{eq:singular data equality} will be rigorously proved once we establish uniqueness of solutions.
\item
 As we will show below, $\Heat[ \Martin [h] , 0 , 0] (t,x)  = \Semi(t) [\Martin[h]] \le C \delta^\gamma(x)$, and hence
		\[
		\E \Big [\Heat[0,0,h] (t, \cdot) \Big ] (x) = h(x), \qquad \forall t > 0.
		\]
\item
	In particular, we have $\Heat[u^\star,0,1](t,x) \equiv u^\star$, and $\E\Heat[0,0,1]=\E\Heat[u^\star,0,1]\equiv1$.

\item Notice that the deduction above formally implies that, as $x \to \eta \in \partial \Omega$,
		\begin{equation*}
			\frac{1}{u^\star(x)}\int_0^t \D_\gamma \p(\sigma, \zeta, x) \diff \sigma \to \delta_\eta (\zeta).
		\end{equation*}
		Since we know formally this is true for the Green kernel
		\begin{equation*}
			\frac{1}{u^\star(x)}\int_0^\infty \D_\gamma \p(\sigma, \zeta, x) \diff \sigma  = \frac{\D_\gamma \G(\zeta,x)}{u^\star(x)} = \frac{\M(x, \zeta)}{u^\star(x)}\to \delta_\eta (\zeta),
		\end{equation*}
		 we only would need to show that for any $t > 0$ we have that
		\begin{equation*}
		\frac{1}{u^\star(x)}\int_t^\infty \p(\sigma, \zeta, x) \diff \sigma \to 0.		
		\end{equation*}
		However, this is not the approach we will take. Instead, we will study the boundary behaviors through a weak formulation in \Cref{sec:general form}.
\end{enumerate}
\end{remark}

	\subsection{Boundary estimates}
	\label{sec:boundary estimates}

	The authors of \cite{abatangelo+gc+vazquez2019} prove some boundary estimates of the type
	\begin{equation*}
	\Green[\delta^\beta] \asymp  \begin{dcases}
	\delta^{\beta + 2s}, & \text{ for }\beta  \in (- \gamma - 1,  \gamma - 2s), \\
	\delta^{\gamma} |\ln \delta|,  & \text{ for }\beta  = \gamma -2 s, \\
	\delta^\gamma, & \text{ for }\beta > \gamma - 2s.
	\end{dcases}
	\end{equation*}
 	The lower bound $\beta>-\gamma-1$  of the range is to guarantee that $\delta^\beta \delta^\gamma \in L^1$. As $\beta \to -\gamma - 1$ we recover $\Green[\delta^\beta ] \sim u^\star \in L^1 (\Omega, \delta^\gamma)$. Since, it is easy to check similarly to Step 1 in the proof of \Cref{thm:main h = 0} that $\Green [f] = \Heat[ \Green[f], f, 0]$, we have that
 	\begin{equation*}
 		\Heat[0, \delta^\beta, 0] = \Heat[ \Green[\delta^\beta] ,  \delta^\beta, 0   ] - \Heat[ \Green[\delta^\beta] ,  0, 0   ] = \Green[\delta^\beta] - \Semi(t) [\delta^\beta].
 	\end{equation*}
 	As we will show in \Cref{sec:heat kernel existence and estimates}, $\Semi(t) [\delta^\beta] \le C \delta^\gamma$ for $t > 0$,  the range of exponents of data $f$ and solutions for the parabolic problem seems to coincide with the one in the elliptic theory.
 	This formalises the intuition of \Cref{rem:satisfies boundary condition}.

\section{One-sided Weyl's law}
\label{sec:weyl}

To formalise the intuition in \Cref{sec:parabolic intuition}, we first prove a result on the growth of eigenvalues.
In \cite{Bonforte+Sire+Vazquez2015}, by bounding the Green's function with the Riesz potential in the dual setting, the authors show that our general family of operators satisfies a Sobolev inequality
\begin{equation}\tag{S}\label{eq:L Sobolev}
\text{There exists } C_S> 0 \text{ such that} \qquad C_S \| u \|_{L^{\alpha} (\Omega)}
\leq
\left( \sum_{k=1}^\infty \lambda_k \langle u , \varphi_k \rangle^2 \right)^{\frac 1 2},
\qquad \forall u\in H^1_\Ls (\Omega).
\end{equation}
The constant $C_S$ is usually called the Sobolev constant of $\Ls$.
This result is shown in  \cite[Theorem 7.5]{Bonforte+Sire+Vazquez2015}  under our assumptions  with exponent \ $\alpha = 2^\star =  \tfrac{2n}{n-2s} > 2$.
Notice that the right-hand side is the energy given equivalently by $ \left( \int_\Omega u \Ls u \right)^{\frac 1 2} = \| \Ls^{\frac 1 2} u \|_{L^2 (\Omega)}$ for any $u \in C_c^\infty (\Omega)$.

In the book by Davies \cite{Davies1989} the existence of this kind of Sobolev inequality is directly linked with the integrability of the heat kernel over large times. Our approach here will be slightly different.%

\begin{theorem}\label{thm:weyl}
	Assume \eqref{eq:G estimate}, \eqref{eq:G delta Linf to delta C} and \eqref{eq:S submarkovian}. Then there exists $c > 0$ depending only on $n,s,|\Omega|$ and the Sobolev constant of $\Ls$ on $\Omega$ such that
	\begin{equation}\label{weyl1}
	\lambda_k \ge c
	{k}^{2s/n}.
	\end{equation}
\end{theorem}

\medskip

\begin{remark}  A stronger form of  Weyl's law of the form
\begin{equation}\label{eq:Weyl full}
\lambda_k=(c+o(1))k^{2s/n}
\end{equation}
holds true for the SFL (since it is a direct consequence of the classical Weyl's law) and for the RFL as was shown already by Blumenthal and Getoor in 1959 \cite{BlumGetoor1959}. A short proof was recently found by Geisinger \cite{Geisinger2014}, see also \cite{Frank-Geisinger2016}. Using pseudodifferential methods, Grubb \cite{Grubb2015} showed Weyl's law for general RFL-type operators.  In contrast, we have not found any such result in the literature for the CFL.
Without assuming that there exists a heat kernel, or that it is smooth, our argument is a generalisation of the one in \cite{Cheng1981} (see also \cite{Schoen1994}) for the Laplace--Beltrami operator on manifolds.
\end{remark} 

\begin{lemma}
	Assume:
	\begin{enumerate}
		\item $\Green : L^2 (\Omega) \to L^2 (\Omega)$ is compact,
		\item The heat semigroup is submarkovian in the sense of \eqref{eq:S submarkovian},
		\item The eigenfunctions are continuous up to the boundary, i.e. $\varphi_k \in C(\bar \Omega)$,
		\item A Sobolev inequality \eqref{eq:L Sobolev} holds for $\Ls$ for some exponent $\alpha>2$.
	\end{enumerate}
	Then, for any $k\in\mathbb{N}$, we have that
	\begin{equation}
	\label{eq:Weyl}
	\lambda_k \ge 2 C_S e^{-2}
	\left(  \frac{k}{|\Omega|}  \right)^{{\frac{\alpha - 2}{\alpha}}}.
	\end{equation}
\end{lemma}

\begin{proof}
	Take $u_0 \in L^2$ such that $\| u_0 \|_{L^1} \le 1$.
	Let $u (t, x) = \Semi(t)[u_0](x)$, so that by \eqref{eq:S L1 cont}, we have that
	\[
	\int_{\Omega}|u(t,x)|\diff x\leq 1.
	\]
	Due to the eigen-decomposition, we have that
	\begin{align*}
	\frac{\partial }{\partial t} \int_\Omega |u(t,x)|^2 \diff x  &= \frac{\partial }{\partial t}\sum_{k=1}^\infty e^{-2 \lambda_k t} \langle u_0, \varphi_k \rangle^2 \\
	& = - \sum_{k=1}^\infty 2\lambda_k e^{-2 \lambda_k t} \langle u_0, \varphi_k \rangle^2 \\
	&= - \sum_{k=1}^\infty 2 \lambda_k \langle u(t,\cdot), \varphi_k \rangle^2 \\
	& \le - 2 C_S \left(
	\int_{\Omega}
	|u(t,x)|^{\alpha}
	\diff x
	\right)^{\frac{2}{\alpha}}.
	\end{align*}
	Using Hölder inequality (or equivalently, the interpolation inequality),
	\begin{equation*}
	\int_\Omega |u(t,x)|^2 \diff x
	\le \left(
	\int_\Omega |u(t,x)|^\alpha \diff x
	\right)^{\frac{1}{\alpha-1}}
	\left(
	\int_\Omega |u(t,x)| \diff x
	\right)^{\frac{\alpha - 2}{\alpha-1}}
	\le \left(
	\int_\Omega |u(t,x)|^\alpha \diff x
	\right)^{\frac{1}{\alpha-1}}.
	\end{equation*}
	Thus
	\begin{equation*}
	\frac{\partial }{\partial t} \int_\Omega |u(t,x)|^2 \diff x \le - 2 C_S \left( \int_\Omega |u(t,x)|^2 \diff x \right)^{ \frac{2\alpha-2}{\alpha} },
	\end{equation*}
	and, integrating,
	\begin{equation*}
	\int_\Omega |\Semi(t)[u_0] (x)|^2 \diff x
	\le
	\left(
	\left(
	\int_\Omega
	|u_0(x)|^2
	\diff x
	\right) ^{ -\frac{\alpha-2}{\alpha} }
	+  \frac {(\alpha-2) 2  C_St}{\alpha}
	\right)^{-\frac{\alpha}{\alpha-2}} .
	\end{equation*}
	Let us fix $y \in \Omega$ and $t > 0$. Let us take the sequence of initial data
	$$u_{0,j} = |B_{\frac 1 j}|^{-1} \chi_{y + B_{\frac 1 j}}.$$
	We have that $\| u_{0,j} \|_{L^1} \le 1$ and $\| u_{0,j} \|_{L^2}  = |B_{\frac1j}|^{-\frac12}  \to \infty$. Therefore,
	\begin{equation}
	\label{eq:weyl estimate 1}
	\int_\Omega |\Semi(t)[u_{0,j}] (x)|^2 \diff x \le
	\left(
	|B_{\frac 1 j} | ^{ \frac{\alpha-2}{\alpha} }
	+  \frac {(\alpha-2)  2  C_St}{\alpha}
	\right)^{-\frac{\alpha}{\alpha-2}}
	\to  \left(  \frac {(\alpha-2) 2 C_St}{\alpha}
	\right)^{-\frac{\alpha}{\alpha-2}}.
	\end{equation}
	Therefore,
	the sequence $\Semi(t)[u_{0,k}]$ as a subsequence converging weakly in $L^2 (\Omega)$.
	
	Let $K(t,\cdot,y) \in L^2 (\Omega)$ be its limit. Since $\Semi(t)$ is self-adjoint
	\begin{equation*}
	\langle \Semi(t) [u_{0,j}] , \varphi_k \rangle = \langle u_{0,j} , \Semi(t) [\varphi_k] \rangle  = e^{-\lambda_k t} |B_{\frac 1 j}|^{-1} \int_{B_{\frac 1 j}} \varphi_k (y + x) \diff x.
	\end{equation*}
	Since $\varphi_k$ is continuous in $x$, passing to the limit
	\begin{equation*}
	\langle K(t, \cdot, y) , \varphi_k \rangle = e^{-\lambda_k t} \varphi_k (y).
	\end{equation*}
	Thus, the weak limit is unique and the whole sequence converges weakly. Due to the weak lower semicontinuity of the norm, we can pass to the limit in \eqref{eq:weyl estimate 1}
	\begin{equation*}
	\int_\Omega K(t,x,y)^2 \diff x
	\le \left(   \frac {(\alpha-2) 2 C_St}{\alpha}
	\right)^{-\frac{\alpha}{\alpha-2}}.
	\end{equation*}
	On the other hand, since $K(t, \cdot, y) \in L^2(\Omega)$ we have that
	\begin{equation*}
	\int_\Omega K(t,x,y)^2 \diff x = \sum_{k=1}^\infty \langle K(t, \cdot, y) , \varphi_k \rangle^2 = \sum_{k=1}^\infty e^{-2\lambda_k t} \varphi_k (y)^2.
	\end{equation*}
	Thus, this series of non-negative functions is uniformly summable due to the estimate. Integrating in $y$
	\begin{equation*}
	\sum_{k=1}^\infty e^{-2\lambda_k t} = \sum_{k=1}^\infty e^{-2\lambda_k t} \int_\Omega \varphi_k (y)^2 \diff y =
	\int_\Omega \int_\Omega K(t,x,y)^2 \diff x \diff y  \le
	\left(   \frac {(\alpha-2) 2 C_St}{\alpha}
	\right)^{-\frac{\alpha}{\alpha-2}} |\Omega| .
	\end{equation*}
	Therefore, we can estimate
	\begin{equation*}
	k e^{-2\lambda_k t}
	\le   \left(   \frac {(\alpha-2) 2 C_St}{\alpha}
	\right)^{-\frac{\alpha}{\alpha-2}}|\Omega|.
	\end{equation*}
	Letting $\lambda_k t = \frac{\alpha}{\alpha-2}$ we recover
	\begin{equation*}
	k e^{-\frac{2\alpha}{\alpha-2}}
	\le \left(   \frac { 2 C_S}{\lambda_k}
	\right)^{-\frac{\alpha}{\alpha-2}} |\Omega|.
	\end{equation*}
	Finally, \eqref{eq:Weyl} follows.
\end{proof}

\section{Existence and estimates of the heat kernel}
	\label{sec:heat kernel existence and estimates}

	\begin{theorem}
		\label{thm:S boundary reg}
		Assume \eqref{eq:G estimate}, \eqref{eq:G delta Linf to delta C} and \eqref{eq:S submarkovian}. Then $\Semi(t)$ is ultracontractive and regularises up to the boundary in the sense that
		\begin{equation}
		\label{eq:S boundary regularity}
		\Semi(t) : L^2 (\Omega) \longrightarrow \delta^\gamma C(\overline \Omega)
		\end{equation}
and
		\begin{equation}
		\label{eq:S boundary regularity dual}
		\Semi(t) : M(\Omega,\delta^\gamma) \longrightarrow L^2 (\Omega)
		\end{equation}
are continuous for all $t>0$. Moreover, the heat kernel $\p(t,x,y)=\Semi(t)[\delta_y](x)$ as well its $\gamma$-normal derivative $\D_\gamma\p(t,x,y)$ exist pointwise in $(0,+\infty)\times\Omega\times\Omega$. Furthermore, for $t>0$, $x,y\in\Omega$ and $\zeta\in\partial\Omega$ we have
\[
0\leq \p(t,x,y)
\leq C(t)\delta(x)^\gamma \delta(y)^\gamma,
    \qquad
\int_{\Omega}
    \p(t,x,y)
\diff x \leq 1,
\]
\[
0\leq \D_\gamma \p(t,\zeta,y)
\leq C(t)\delta(y)^\gamma.
\]
	\end{theorem}

    In the literature, the reader will find many equivalent ways of proving this result. We list a few below.
    \normalcolor
    \begin{lemma}
	    	\label{lem:S boundary reg}
    	Assume \eqref{eq:G estimate} and \eqref{eq:G delta Linf to delta C}. Then, 
    	the following are equivalent:
    	\begin{enumerate}
    		\item \label{it:S boundary reg 1}
    		We have \eqref {eq:S boundary regularity}.
    		\item \label{it:S boundary reg 2}
            We have \eqref {eq:S boundary regularity dual}.		

    		\item \label{it:S boundary reg 3}
    		We have that $\Semi(t)$ is given %
    by the integral operator in \eqref{eq:Semi given kernel} with kernel $\p(t,x,y) = \Semi(t) [\delta_y] (x)$
    and is intrinsic ultracontractive:
    		for every $t>0$, there exists a constant $C (t) > 0$ such that
    		\begin{equation}
    		\label{eq:S estimate}
    			\p(t,x,y) \le C(t) \delta(x)^\gamma \delta(y)^\gamma.
    		\end{equation}
    		\item \label{it:S boundary reg 4}
    		For every $t>0$ we have that
    		\begin{equation*}
    		 \sum_{k=1}^\infty e^{- 2 \lambda_k t} = \int_\Omega \int_\Omega \p(t,x,y)^2 \diff y \diff x  < + \infty.
    		\end{equation*}

    		\item \label{it:S boundary reg 5}
    		For every $t, w > 0$ we have that
    		\begin{equation*}
    		\sum_{k=1}^\infty e^{- \lambda_k t} \lambda_k^w < +\infty.
    		\end{equation*}
    	\end{enumerate}
    	If any of the above holds, for $t_0 > 0$ fixed, as $m \to +\infty$ we have that
    		\begin{equation*}
    		\int_\Omega \int_\Omega  \left|   \p(t_0,x,y) - \sum_{k=1}^m e^{-\lambda_k t_0} \varphi_k (x) \varphi_k (y)  \right|^2 \diff y \diff x = \sum_{k=m+1}^\infty e^{-2 \lambda_k t} \to 0.
    		\end{equation*}
    \end{lemma}

\begin{proof}[Proof of \Cref{thm:S boundary reg} assuming \Cref{lem:S boundary reg}]
	 Properties \eqref{eq:S boundary regularity} and \eqref{eq:S boundary regularity dual} are immediate consequences of \Cref{thm:weyl}, as \Cref{it:S boundary reg 5} in \Cref{lem:S boundary reg} holds true. As a result, $\p(t,\cdot,y)=\Semi(t)[\delta_y]$ exists in $\delta^\gamma C(\overline{\Omega})$ and thus the limit \eqref{eq:D gamma S defn} is well-defined and uniform in $\zeta \in \partial \Omega$ for every $t$ fixed since
		\begin{equation*}
		\D_\gamma \p (t,\zeta,y) = \D_\gamma \Semi(t) [\delta_y] (\zeta).
		\end{equation*}
The remaining properties follow from a concentration argument. We have that
	    \begin{equation*}
	    	 \Semi (t) \left[  |B_{\frac1j}|^{-1}\chi_{y + B_{\frac1j}} \right] \longrightarrow  \p(t,\cdot,y) \qquad \text{in } \delta^\gamma C(\overline \Omega).
	    \end{equation*}
Due to \eqref{eq:S L1 cont} we have that
	    \begin{equation*}
	    	\int_\Omega \p (t,x,y) \diff x \le 1.
	    \end{equation*}
	    Due to \eqref{eq:S submarkovian} and \eqref{eq:S estimate} we have that
\[
0 \leq \p(t,x,y) \leq C(t)\delta(x)^\gamma \delta(y)^\gamma.
\]
Dividing the inequalities by $\delta(x)^\gamma$ and passing to the limit $x\to\zeta$, we conclude that
		\begin{equation*}
		0 \le \D_\gamma \p (t, \zeta,y)   \le C(t) \delta(y)^\gamma.
    \qedhere
		\end{equation*}
\end{proof}

	    \begin{remark}
\label{rmk:rigorous parabolic}
	    	We point out several facts:
	    \begin{enumerate}
        \item \label{it:C}
            Without loss of generality, we take in \eqref{eq:S estimate} the optimal $C(t)$, i.e.
            \[
            C(t)
            =\sup_{x,y\in\Omega}
                \dfrac{\p(t,x,y)}{\delta(x)^\gamma\delta(y)^\gamma}
            =\sup_{y\in\Omega}
                \dfrac{
                    \norm[C(\overline\Omega)]{
                        \Semi(t)[\delta_y]/\delta^\gamma
                    }
                }{
                    \norm[M(\Omega,\delta^\gamma)]{
                        \delta_y
                    }
                }
            \leq
                \norm[
                    \mathcal{L}(
                        M(\Omega,\delta^\gamma),
                        \delta^\gamma C(\overline\Omega)
                    )
                ]{
                    \Semi(t)
                }.
            \]
            We point out that $\norm[M(\Omega,\delta^\gamma)]{
            	\delta_y} = \delta(y)^\gamma$.
		\item \label{it:integrability C} From the proof and \Cref{it:C}, we recover that
		\begin{align*}
			C(t) &\le \| \Semi(t) \|_{\mathcal L (M(\Omega, \delta^\gamma), \delta^\gamma C(\overline \Omega) )} \\
			& \le \| \Semi(t/2) \|_{\mathcal L (L^2 (\Omega), \delta^\gamma C(\overline \Omega) )} \| \Semi(t/2) \|_{\mathcal L ( M(\Omega, \delta^\gamma) , L^2 (\Omega)) } \\
			&= \| \Semi(t/2) \|_{\mathcal L (L^2 (\Omega), \delta^\gamma C(\overline \Omega) )}^2 \\
			& \le C \left( \sum_{k=1}^\infty e^{-\lambda_k t/2} \lambda_k^{w} \right)^{2},
		\end{align*}
		where $w$ is such that $\Green^w : L^2 (\Omega) \to \delta^\gamma C(\overline \Omega)$. Due to \eqref{eq:G delta Linf to delta C} we have that $w \le w' + 1$ where  $\Green^{w'} : L^2 (\Omega) \to \delta^\gamma L^\infty(\Omega)$ and $w'$  can be recovered from \eqref{eq:G estimate}.
		With an estimate of this value $w$ and the one-sided Weyl's law, we could recover an integrability estimate of $C(t)$.

	    \item %
Using \eqref{eq:G from S} with $|B_{\frac1j}|^{-1} \chi_{y + B_{\frac1j}} \in L^2(\Omega)$,
		\begin{equation*}%
		\begin{split}
		    \int_{0}^{\infty}
		        \Semi (t) \left[
		            |B_{\frac1j}|^{-1}
		            \chi_{y + B_{\frac1j}}
		        \right]
                 (x)
		    \diff t
		=
		    \Green\left[
		            |B_{\frac1j}|^{-1}
		            \chi_{y + B_{\frac1j}}
		    \right]  (x) .\\
		\end{split}\end{equation*}
		Since the right hand side is bounded for any $ x\neq y$ in $\Omega$ and $j$ large enough, by Dominated Convergence Theorem, we can pass to the limit $j\to\infty$ to obtain
		\[
		\G(x,y)=\int_0^\infty \p(t,x,y)\diff t.
		\]

		\item
    	Besides this eigenvalue theory, many results on regularising 	properties of semigroups in the linear and nonlinear setting are known. We point the reader to the exposition
\cite{Coulhon2016}.

   		\item
    	Since $\Semi(0) = I$ which cannot be written as integration with a bounded kernel, we know that  $C(t)  \to +\infty$ as $t \to 0$.
    	
    \item
    	If the energy spaces are embedded into higher regularity spaces, i.\,e., there exists $w,k\in\mathbb{N}$ such that $H_\Ls^{2w}(\Omega) \hookrightarrow C^k(\Omega)$, then
    	this translates into regularity of $\Semi(t)$, $t>0$.
        Indeed, we have that
        $\Green^w: L^2(\Omega) \to H_\Ls^{2w}(\Omega) \hookrightarrow C^k(\Omega)$, and the implication \textbf{\Cref{it:S boundary reg 5} $\implies$ \Cref{it:S boundary reg 1}} in the proof shows that $\Semi(t):M(\Omega,\delta^\gamma) \longrightarrow C^k(\Omega) \cap \delta^\gamma C(\overline{\Omega})$.
    \end{enumerate}
    \end{remark}

    \begin{proof}[Proof of \Cref{lem:S boundary reg}]
    \noindent\textbf{\Cref{it:S boundary reg 1} $\implies$ \Cref{it:S boundary reg 2}}
    Since $\Semi(t)$ is self-adjoint, \eqref{eq:S boundary regularity} implies by duality that
       \begin{equation*}
   \Semi(t) : M (\Omega, \delta^\gamma) \to L^2( \Omega).
   \end{equation*}
   Furthermore,
   the bootstrap satisfies $\Semi(t) = \Semi(t/2) \Semi(t/2): M(\Omega, \delta^\gamma) \to L^2 (\Omega) \to \delta^\gamma C(\overline \Omega)$.

   \noindent\textbf{\Cref{it:S boundary reg 2} $\implies$ \Cref{it:S boundary reg 3}} In particular we can recover the kernel as
   \begin{equation*}
   	\p(t,x,y) = \Semi(t) [\delta_y] (x)
   \end{equation*}
   where $\delta_y$ is the Dirac delta at $y$. Furthermore
   \begin{equation*}
   	\frac {\p(t,x,y)}{\delta(y)^\gamma} =  \Semi(t) \left [ \frac{ \delta_y } {\delta(y)^\gamma} \right] (x) \le C(t) \delta(x)^\gamma.
   \end{equation*}

	\noindent\textbf{\Cref{it:S boundary reg 3} $\implies$ \Cref{it:S boundary reg 4}} Let $t_0 > 0, x_0 \in \Omega$ be fixed. The function $\p(t_0,x_0,\cdot)  \in L^\infty (\Omega) \subset L^2 (\Omega)$. Therefore, for a.e. $y \in \Omega$,
	\begin{equation*}
		\p(t_0,x_0,y) = \sum_{k=1}^\infty \langle \p (t_0, x_0, \cdot) , \varphi_k \rangle \varphi_k (y).
	\end{equation*}
	Notice that
	\begin{equation*}
		\langle \p (t_0, x_0, \cdot) , \varphi_k \rangle = \int_\Omega \p(t_0, x_0, y) \varphi_k (y) \diff y  = \Semi(t_0) [\varphi_k] (x_0) = e^{-\lambda_k t_0} \varphi_k (x_0).
	\end{equation*}
	Furthermore, the series converges in $L^2 (\Omega)$ so
	\begin{equation*}
		\int_\Omega \left|   \p(t_0,x_0,y) - \sum_{k=1}^m e^{-\lambda_k t_0} \varphi_k (x_0) \varphi_k (y)  \right|^2 \diff y \to 0 \qquad \text{ as } m \to+\infty.
	\end{equation*}
	Computing the norm of the finite sums and passing to the limit
	\begin{equation*}
		\int_\Omega \p(t_0,x_0,y)^2 \diff y = \sum_{k=1}^\infty e^{-2 \lambda_k t_0} \varphi_k(x_0)^2 .
	\end{equation*}
	Let us look at the sequence of functions
	\begin{equation*}
		u_m (x) \defeq  \sum_{k=1}^m e^{-2 \lambda_1 t_0} \varphi_k(x)^2 \le  \int_\Omega \p(t_0,x,y)^2 \diff y \defbyeq u(x)
	\end{equation*}
	The sequence $u_m$ is monotone non-decreasing sequence of non-negative functions such that $u_m \to u$ pointwise. By the Monotone Convergence Theorem
	\begin{equation*}
		\sum_{k=1}^m e^{-2 \lambda_1 t_0}  = \int_\Omega u_m (x) \diff x \to \int_\Omega u(x) \diff x = \int_\Omega \int_\Omega \p(t_0,x,y)^2 \diff y \diff x.
	\end{equation*}
	Since $\p(t_0,\cdot,\cdot)$ is bounded on $\Omega\times\Omega$, the infinite series converges.

	\noindent\textbf{\Cref{it:S boundary reg 4} $\implies$ \Cref{it:S boundary reg 5}.} We simply need to point that, for every $t, w > 0$, there exists an $\underline \lambda(t,w)$ such that, if $\lambda \ge \underline \lambda(t,w)$
	\begin{equation*}
		\lambda^w \le e^{\lambda \frac t 2}.
	\end{equation*}
	Therefore
	\begin{align*}
		\sum_{k=1}^\infty e^{- \lambda_k t} \lambda_k^w &= \sum_{\lambda_k < \underline \lambda(t,w) } e^{- \lambda_k t} \lambda_k^w +  \sum_{\lambda_k \ge \underline \lambda(t,w) } e^{- \lambda_k t} \lambda_k^w \\
		&\le  \sum_{\lambda_k < \underline \lambda(t,w) } e^{- \lambda_k t} \lambda_k^w +  \sum_{\lambda_k \ge \underline \lambda(t,w) } e^{- \lambda_k \frac t 2}
	\end{align*}
	Since the sequence of eigenvalues diverges, the first sum has a finite number of elements. The second term is finite by \Cref{it:S boundary reg 4}.
	
	\noindent\textbf{\Cref{it:S boundary reg 5} $\implies$ \Cref{it:S boundary reg 1}.} From the elliptic theory (see \cite{abatangelo+gc+vazquez2019,chan+gc+vazquez2020}) we know by \eqref{eq:G estimate} and \eqref{eq:G delta Linf to delta C} that there exists $w\in \mathbb N$ such that $\Green^w : L^2 (\Omega) \to \delta^\gamma C(\overline \Omega)$. Since, for an eigenfunction $\varphi_k = \lambda_k^w \Green^w [\varphi_k]$ we have that
	\begin{equation}\label{eq:varphi-bound}
	\left\| \frac{\varphi_k}{\delta^\gamma} \right \|_{C ( \overline\Omega )}
\le  C \lambda_k^w \| \varphi_k \|_{L^2 (\Omega)} = C \lambda_m^w,
	\end{equation}
for $C$ independent of $m$.
This is the philosophy behind the proof of \cite[Proposition 3.1]{Fdz-Real+RosOton2016} (written there for the RFL). Following the same idea as \cite[Theorem 1.1]{Fdz-Real+RosOton2016} we can apply \eqref{eq:regularity of the eigenf} to deduce
	\begin{equation*}
	\left\| \frac{\Semi(t)[u_0]}{\delta^\gamma} \right\|_{C(\overline \Omega)}  \le \sum_{k=1}^\infty e^{-\lambda_k t} |\langle u_0, \varphi_k \rangle|	\left\| \frac{\varphi_k}{\delta^\gamma} \right\|_{C(\overline \Omega)} \le C \| u_0 \|_{L^2 (\Omega)} \sum_{k=1}^\infty e^{-\lambda_k t} \lambda_k^w .
	\end{equation*}
    By the assumption of {\Cref{it:S boundary reg 5}}, the eigenvalues grow so fast that the  last series is summable. Since
	\begin{equation*}
\set{
	u_0 (x) = \sum_{k=1}^m c_k \varphi_k (x)
:m\in \mathbb N,\,c_k\in\re
}
	\end{equation*}
	is a dense set in $L^2 (\Omega)$ such that $\Semi(t) [u_0] \in \delta^\gamma C(\overline \Omega)$, the inequality above guaranties that $\Semi(t)[u_0] / \delta^\gamma$ can be  approximated  uniformly by continuous functions, and is, therefore, continuous.
\end{proof}

\begin{remark}
	Notice that this argument is valid as long as $\p(t,x, \cdot) \in L^2(\Omega)$, and can be extended to $\G$ so long as $\G (x, \cdot) \in L^2 (\Omega)$. Due to \eqref{eq:G estimate} this would require that  $n<4s$.
\end{remark}

To conclude this section we state some estimates of the heat kernel as well as its $\gamma$-normal derivative in the general setting.

\begin{theorem}\label{thm:S estimate}
Assume \eqref{eq:G estimate}, \eqref{eq:G delta Linf to delta C} and \eqref{eq:S submarkovian}. Then the heat kernel $\p(t,x,y)$ satisfies the following estimates.
\begin{enumerate}
\item \label{it:S estimate 1} For all $t>0$ and $x,y\in\Omega$,
\[
0\leq \p(t,x,y)
\leq Ct^{-\frac{n}{2s}}.
\]
\item If $\gamma<2s$, then for all $t>0$, $x,y\in\Omega$ and $\zeta\in\partial\Omega$,
\[
0\leq \p(t,x,y)
\leq Ct^{-\frac{n}{2s-\gamma}}
    \delta(x)^\gamma
    \delta(y)^\gamma,
\]
\[
0\leq \D_\gamma\p(t,\zeta,y)
\leq Ct^{-\frac{n}{2s-\gamma}}
    \delta(y)^\gamma.
\]
\item For any $\ee>0$ there exists a large $T>0$ such that for all $t\geq T$, $x,y\in\Omega$ and $\zeta\in\partial\Omega$,
\begin{equation}
\label{eq:asymptotic S large times}
(1-\ee)
    e^{-\lambda_1 t}
    \varphi_1(x)
    \varphi_1(y)
\leq
\p(t,x,y)
\leq (1+\ee)
    e^{-\lambda_1 t}
    \varphi_1(x)
    \varphi_1(y),
\end{equation}
\[
(1-\ee)
    e^{-\lambda_1 t}
    \D_\gamma\varphi_1(\zeta)
    \varphi_1(y)
\leq
\D_\gamma\p(t,\zeta,y)
\leq
(1+\ee)
    e^{-\lambda_1 t}
    \D_\gamma\varphi_1(\zeta)
    \varphi_1(y).
\]
\end{enumerate}
\end{theorem}

\begin{proof}

\begin{enumerate}
\item Recall \eqref{eq:L Sobolev},
\[
C_S\norm[L^{2^*}(\Omega)]{u}^2
\leq
\int_{\Omega}
    u\Ls u
\diff x,
    \qquad \forall u\in H^1_\Ls (\Omega),
\]
for $2^*=\frac{2n}{n-2s}=\frac{2(n/s)}{(n/s)-2}$. Applying \cite[Theorem 2.4.2]{Davies1989} (with $\mu=\frac{n}{s}>2$), the Sobolev inequality is equivalent to an $L^2 \to L^\infty$ bound of the form
\[
\norm[L^\infty(\Omega)]
    {\Semi(t)[u_0]}
\leq Ct^{-\frac{n}{4s}}
    \norm[L^2(\Omega)]
        {u_0},
\qquad \forall u_0\in L^2(\Omega).
\]
Since $\Semi(t)$ is self-adjoint on $L^1\cap L^2(\Omega)$, we also have the duality bound from $L^1 \to L^2$,
\[\begin{split}
\norm[L^2(\Omega)]
    {\Semi(t)[u_0]}
&=\sup_{
    \norm[L^2(\Omega)]{v_0}\leq1
}
    \angles{u_0,\Semi(t)[v_0]}
\leq
    \sup_{
        \norm[L^2(\Omega)]{v_0}\leq 1
    }
    \norm[L^1(\Omega)]{u_0}
    \bigl(
        Ct^{-\frac{n}{4s}}
        \norm[L^2(\Omega)]{v_0}
    \bigr)\\
&\leq Ct^{-\frac{n}{4s}}
    \norm[L^1(\Omega)]
        {u_0},
\end{split}\]
for all $u_0\in L^1(\Omega)$. Factorizing $\Semi(t)=\Semi(\frac{t}{2})\Semi(\frac{t}{2})$,
\[
\norm[L^\infty(\Omega)]{
    \Semi(t)[u_0]
}
\leq
    Ct^{-\frac{n}{4s}}
    \norm[L^2(\Omega)]{
        \Semi(\tfrac{t}{2})[u_0]
    }
\leq
    Ct^{-\frac{n}{2s}}
    \norm[L^1(\Omega)]{u_0},
        \qquad \forall u_0\in L^1(\Omega).
\]
Now we consider $u_0=u_j$ for the sequence $u_j=|B_{\frac1j}|^{-1}\chi_{y + B_{\frac1j}}$ which satisfies $u_j\geq 0$, $\norm[L^1(\Omega)]{u_j}=1$ and $u_j\to\delta_y$, we get
\[
0\leq \p(t,x,y) \leq Ct^{-\frac{n}{2s}},
    \qquad \forall t>0,\,x,y\in\Omega.
\]
\item
Observing from \cite[Lemma 4.2.2]{Davies1989} and \eqref{eq:S estimate} that the heat semigroup for the conjugate operator
$\overline{\Ls}=\varphi_1\Ls(\varphi_1^{-1}\cdot)$ is ultracontractive with kernel $\overline{\p}(t,x,y)=\varphi_1(x)^{-1}\p(t,x,y)\varphi_1(y)^{-1}$, it suffices to establish a Sobolev inequality for $\overline{\Ls}$ similarly to \eqref{eq:L Sobolev}.
Notice that $\overline \varphi_k = \varphi_1 \varphi_k$ are eigenfunctions of this operator, associated to the eigenvalues $\lambda_k$ of $\Ls$. These functions $\varphi_k$ form a basis of $\varphi_1 L^2 (\Omega)$.  Hence, the associated Dirichlet problem has a well-defined Green kernel
\begin{equation*}
	\overline \Green [f] = \varphi_1 \Green\left[ \frac{f}{\varphi_1} \right], \qquad \forall f \in \varphi_1 L^2 (\Omega).
\end{equation*}
This kernel can be suitably extended.

In fact, for any $g\in C_c^\infty (\Omega)$, using \eqref{eq:G estimate} and Hardy--Littlewood--Sobolev inequality we have
\[\begin{split}
    \left(
        \int_{\Omega}
            g \overline \Green [g]
        \diff x
    \right)^{\frac12}
&=
    \left(
        \int_{\Omega}
            g \varphi_1
            \Green[\varphi_1^{-1} g]
        \diff x
    \right)^{\frac12}\\
&\leq
    C\left(
        \int_{\Omega}
            g(x)\varphi_1(x)
            \int_{\Omega}
                \dfrac{
                    \varphi_1(y)
                }{
                    |x-y|^{n-2s+\gamma}
                }
                \varphi_1(y)^{-1}g(y)
            \diff y
        \diff x
    \right)^{\frac12}\\
&\leq
    C\norm[L^\infty(\Omega)]{\varphi_1}^{\frac12}
    \left(
        \int_{\Omega}\int_{\Omega}
            \dfrac{
                g(x)g(y)
            }{
                |x-y|^{n-2s+\gamma}
            }
        \diff y\diff x
    \right)^{\frac12}\\
&\leq
    C\norm[L^\infty(\Omega)]{\varphi_1}^{\frac12}
    \left(
        \int_{\Omega}
            \abs{
                I_{s-\gamma/2}[g_0]
            }^2
        \diff x
    \right)^{\frac12}\\
&\leq
    C\norm[L^\infty(\Omega)]{\varphi_1}^{\frac12}
    \norm[L^{\frac{2n}{n+2s-\gamma}}]{g},
\end{split}\]
where $g_0$ denotes the extension of $g$ by $0$ onto $\re^n$ and $I_{s-\gamma/2}$ denotes the Riesz potential of order $s-\gamma/2>0$. Now the duality argument in \cite[Section 7.8]{Bonforte+Sire+Vazquez2015} implies that \eqref{eq:L Sobolev} holds for $\overline{\Ls}$ with $\alpha  =\bigl(\frac{2n}{n+2s-\gamma}\bigr)' =\frac{2n}{n-2s+\gamma}=\frac{2(\frac{n}{s-\gamma/2})}{(\frac{n}{s-\gamma/2})-2}$.
Then, similarly to \Cref{it:S estimate 1}, \cite[Theorem 2.4.2]{Davies1989} applies with $\mu=\frac{n}{s-\gamma/2}>2$ and we conclude that
\[
0\leq \overline{\p}(t,x,y)
\leq Ct^{-\frac{n}{2s-\gamma}},
\]
as desired. The estimate for the $\gamma$-normal derivative follows immediately by dividing by $\delta(x)^\gamma$ and passing to the limit $x\to\zeta$.

\item
For the last estimate we notice from \Cref{thm:weyl} and \eqref{eq:varphi-bound} that
\[
\p(t,x,y)
=\sum_{m=1}^{\infty}
    e^{-\lambda_m t}
    \varphi_m(x)
    \varphi_m(y)
\]
is absolutely convergent for any $t > 0$,
and
\[\begin{split}
\abs{
    \dfrac{
        \p(t,x,y)
    }{
        e^{-\lambda_1 t}
        \varphi_1(x)
        \varphi_1(y)
    }
    -1
}
&\leq
    \sum_{m=2}^{\infty}
        e^{(\lambda_1 - \lambda_m)t }
        \dfrac{\varphi_m(x)}{\varphi_1(x)}
        \dfrac{\varphi_m(y)}{\varphi_1(y)}\\
&\leq
    Ce^{(\lambda_1-\lambda_2)t/2}
    \sum_{m=2}^{\infty}
        e^{(\lambda_1 -\lambda_m)t/2}
        \lambda_m^{2w}\\
&\leq
    Ce^{(\lambda_1 -\lambda_2)T/2},
\end{split}\]
which can be smaller than $\varepsilon$  by choosing $T$ large since $\lambda_1$ is simple.  Again, the estimate for $\D_\gamma\p$ is a direct consequence of \eqref{eq:asymptotic S large times}.  \qedhere
\end{enumerate}
\end{proof}

\begin{remark}
	\label{rem:Hopf for t large}
	Notice that \eqref{eq:asymptotic S large times} yields that there exists $T > 0$ and $C > 0$ such that we have that for $u_0 \ge 0$
	\begin{equation*}
		\frac 1 {C} \int_\Omega u_0 \delta^\gamma \le \frac{ \Semi(t) [u_0] } {\delta^\gamma} \le {C} \int_\Omega u_0 \delta^\gamma \qquad \forall t \ge T.
	\end{equation*}
	Hence, if $0 \le u_0 \notin L^1 (\Omega, \delta^\gamma)$ then $\Semi(t) [u_0] \equiv +\infty$ for $t$ large enough.
\end{remark}

\section{A candidate for solution with semigroup representation}
\label{sec:general form}

    The heuristics in \Cref{sec:parabolic intuition} suggests that the general form of the solution of \eqref{eq:main}--\eqref{eq:init}  is given by
	\begin{equation}
	\label{eq:general form of the solution}
		\Heat[u_0, f, h ] (t,x) \defeq \Semi (t) [u_0] (x) + \int_0^t \Semi(t-\sigma) [f(\sigma, \cdot)] (x) \diff \sigma+  \int_0^t \int_{\partial \Omega} \D_\gamma \p (t-\sigma, \zeta, x) h(\sigma, \zeta) \diff \zeta \diff \sigma.
	\end{equation}
	We will first study the continuity properties with respect to $u_0$ and $f$, and the properties with respect to $h$ will come later (see \Cref{sec:proof main result general}).
	In some situations we will use the letter $\phi$ instead of $f$ because in application this continuity result is applied to the test function $\phi\in L^\infty(0,T;\delta^\gamma L^\infty(\Omega))$ in the weak-dual formulation \eqref{eq:weak phi} below.

	\subsection{Functional space continuity of \texorpdfstring{$\Heat$}{H} when \texorpdfstring{$h = 0$}{h=0}}

	\begin{theorem}
		\label{thm:continuity of H}
         Assume \eqref{eq:G estimate}, \eqref{eq:G delta Linf to delta C} and \eqref{eq:S submarkovian}. 
        Then we have that:

        \begin{enumerate}
		\item $
		\Heat : L^2 (\Omega) \times L^2 ((0,T) \times \Omega )  \times\{0\} \longrightarrow L^2 ((0,T) \times \Omega ) )
		$
		is continuous with estimate
		\begin{equation*}
		\| \Heat[u_0, f, 0] (t, \cdot) \|_{L^2 (\Omega)}^{2} \le e^{-\lambda_1 t} \|u_0\|_{L^2 (\Omega)}^{2}
		+ \int_0^t \| f (\sigma, \cdot) \|_{L^2 (\Omega)}^{2} \diff \sigma.
		\end{equation*}
If $u_0, f \ge 0$ then $\Heat[u_0,f,0] \ge 0$.
		\item
		$
		\Heat:\set{0}\times \delta^\gamma L^\infty((0,T)\times\Omega) \times \set{0}
		\longrightarrow \delta^\gamma L^\infty((0,T)\times \Omega)
		$
		is continuous with estimate
		\begin{equation}
		\label{eq:H estimates proof 1}
			\left \| \dfrac{
				\Heat[0,\phi,0](t,\cdot )
			}{
				\delta^\gamma
			}
			\right\| _{L^\infty (\Omega)}
    \le C \int_0^t e^{-\lambda_1 \sigma} \diff \sigma
    \norm[L^\infty((0,T)\times\Omega)]
			{\dfrac{\phi}{\delta^\gamma}}.
		\end{equation}
     If $\phi\geq 0$, then $\Heat[0,\phi,0]\geq 0$.
		\item $\D_\gamma \Heat:\set{0}\times \delta^\gamma L^\infty((0,T)\times \Omega) \times \set{0}
		\longrightarrow L^\infty((0,T) \times \partial  \Omega)$ is continuous and \eqref{eq:D gamma H 0 f 0} holds. Moreover, there holds the estimate
		\begin{equation}
		\label{eq:H estimates proof 2}
			\left \|
                \D_\gamma\Heat[0,\phi,0](t,\cdot)
			\right\| _{L^\infty (\partial\Omega)}
    \le C \int_0^t e^{-\lambda_1 \sigma} \diff \sigma
    \norm[L^\infty((0,T)\times\Omega)]
			{\dfrac{\phi}{\delta^\gamma}}.
		\end{equation}
    If $\phi\geq 0$, then $\D_\gamma \Heat[0,\phi,0]\geq 0$.
	\end{enumerate}
	\end{theorem}

\begin{example}
It is instructive to notice that when $\phi(t,x)=\varphi_1(x)$, we have for $x\in\Omega$ and $\zeta\in\partial\Omega$,
\[
\Heat[0,\varphi_1,0](t,x)
=\left(
\int_{0}^{t}
    e^{-\lambda_1 \sigma}
\diff\sigma
\right)\varphi_1(x),
    \qquad
\D_\gamma\Heat[0,\varphi_1,0](t,\zeta)
=\left(
\int_{0}^{t}
    e^{-\lambda_1 \sigma}
\diff\sigma
\right)
\dfrac{\varphi_1}{\delta^\gamma}(\zeta).
\]
\end{example}

\begin{proof}[Proof of \Cref{thm:continuity of H}]
Due to the linearity and the non-negativity of the kernels, it suffices to work with $\phi \ge 0$.
\begin{enumerate}
\item
The $L^2$ theory follows directly from the semigroup representation and eigendecomposition.
\item
By Duhamel's formula \eqref{eq:Duhamel} and the fact that $\p\geq0$ (which follows from \eqref{eq:S submarkovian}), we have
\begin{equation*}
\begin{split}
\dfrac{
        \Heat[0,\phi,0](t,x)
    }{
        \varphi_1(x)
    }
&=\int_0^t\int_{\Omega}
    \dfrac{
        \p(\sigma,x,y)
    }{
        \varphi_1(x)
    }
     \phi(t-\sigma,y)
	\diff y\diff \sigma\\
&\leq
    \norm[L^\infty((0,T)\times\Omega)]
        {\dfrac{\phi}{\varphi_1}}
    \int_0^t\int_{\Omega}
        \dfrac{
            \p(\sigma,x,y)
        }{
            \varphi_1(x)
        }
        \varphi_1(y)
    \diff y\diff \sigma\\
&=
    \norm[L^\infty((0,T)\times\Omega)]
        {\dfrac{\phi}{\varphi_1}}
    \int_0^t
        e^{-\lambda_1 \sigma}
    \diff \sigma. \\
   \end{split}
\end{equation*}
This proves \eqref{eq:H estimates proof 1}.
\item
In order to check that $\D_\gamma$ is well-defined and bounded we fix $t > 0$ and $\zeta \in \partial \Omega$ and, for any $x\in\Omega$, we split 
\begin{align*}
\frac{\Heat[0,\phi,0] (t,x)}{\delta (x)^\gamma}
	= \int_{\ee}^{t}\int_{\Omega}
        \dfrac{
            \p(\sigma,x,y)
        }{
            \delta(x)^\gamma
        }
        \phi(t-\sigma,y)
    \diff y\diff\sigma
    + \int_{0}^{\ee}\int_{\Omega}
    \dfrac{
    	\p(\sigma,x,y)
    }{
    	\delta(x)^\gamma
    }
    \phi(t-\sigma,y)
    \diff y\diff\sigma.
\end{align*}
The second term is controlled by \eqref{eq:H estimates proof 1} we have that
\begin{equation*}
	 \int_{0}^{\ee}\int_{\Omega}
	\dfrac{
		\p(\sigma,x,y)
	}{
		\delta(x)^\gamma
	}
	\phi(t-\sigma,y)
	\diff y\diff\sigma
	\le
	C \norm[L^\infty((0,T)\times\Omega)]
	{\dfrac{\phi}{\delta^\gamma}}
	\lambda_1^{-1} (1 - \ee^{-\lambda_1 \ee}) = \vcentcolon \omega(\ee) .
\end{equation*}
Therefore, we have that
\begin{equation*}
\int_{\ee}^{t}\int_{\Omega}
\dfrac{
	\p(\sigma,x,y)
}{
	\delta(x)^\gamma
}
\phi(t-\sigma,y)
\diff y\diff\sigma \le \frac{\Heat[0,\phi,0] (t,x)}{\delta (x)^\gamma}
\le \int_{\ee}^{t}\int_{\Omega}
\dfrac{
	\p(\sigma,x,y)
}{
	\delta(x)^\gamma
}
\phi(t-\sigma,y)
\diff y\diff\sigma +  \omega(\ee) .
\end{equation*}

The first term, on the other hand, admits a limit as $x \to \zeta \in \partial\Omega$ in view of \eqref{eq:S boundary regularity} so
\begin{multline*}
\int_{\ee}^{t}\int_{\Omega}
\D_\gamma
\p(\sigma,x,y)
\phi(t-\sigma,y)
\diff y\diff\sigma
\\
\le
\liminf_{x \to \zeta}
\frac{\Heat[0,\phi,0] (t,x)}{\delta (x)^\gamma}
\le
\limsup_{x \to \zeta}
\frac{\Heat[0,\phi,0] (t,x)}{\delta (x)^\gamma}  \\
\le
\int_{\ee}^{t}\int_{\Omega}
\D_\gamma
\p(\sigma,x,y)
\phi(t-\sigma,y)
\diff y\diff\sigma
+  \omega(\ee) .
\end{multline*}
As $\ee \to 0$ we can apply the monotone convergence theorem to show that the $\limsup$ and $\liminf$ coincide, and \eqref{eq:H estimates proof 1} to give an estimate on the $\limsup$,
\begin{equation*}
\lim_{x \to \zeta}
\frac{\Heat[0,\phi,0] (t,x)}{\delta (x)^\gamma}
= \int_{0}^{t}\int_{\Omega}
\D_\gamma
\p(\sigma,x,y)
\phi(t-\sigma,y)
\diff y\diff\sigma \leq C
\int_{0}^{t}
e^{-\lambda_1 \sigma}
\diff\sigma
\norm[L^\infty((0,T)\times\Omega)]
{\dfrac{\phi}{\delta^\gamma}}. \qedhere
\end{equation*}
\end{enumerate}
\end{proof}

\subsection{Compactness theory}

	We introduce a second estimate, that will allow us to pass to the limit
	
	\begin{lemma}[Space-time uniform integrability]
		\label{lem:uniform integrability}
		Let  $A \subset \Omega$, $t_0 \ge 0$ and $h>0$. Then
		\begin{equation}
		\int_{t_0}^{t_0+h} \int_A |\Heat[u_0,f,0]  | \delta^\gamma  \le \omega_T (h) \omega (|A|)   \left( 	\int_\Omega |u_0 (x)| \delta^\gamma \diff x + \int_\Omega \int_0^T | f(t, x) | \delta^\gamma \diff t \diff x \right) .
		\end{equation}
		In fact, we can take $\omega_T(h)=Ch^{\frac 1 2}$ and $\omega(|A|)\le C |A|^{\frac 1 {2q}}$ for some $q > 1$ large. In particular, taking $A=\Omega$, $t_0=0$ and $h=T$,
		$$
		\Heat:L^1 (\Omega, \delta^\gamma)\times L^1(0,T;L^1 (\Omega, \delta^\gamma)) \times \set{0}
		\longrightarrow L^1(0,T;L^1 (\Omega, \delta^\gamma)).
		$$
	\end{lemma}
	\begin{proof}First assume $u_0, f \ge 0$.
		
		\noindent \textbf{Step 1. Time compactness.} We take $\phi(t,x) = \chi_{[t_0,t_0+h]} (t) \varphi_1 (x)$. We can solve directly that
		\begin{equation*}
		\Heat[0, \phi, 0] (t,x) = \varphi_1 (x) \int_0^t e^{-\lambda_1 (t-\sigma)}  \chi_{[t_0,t_0+h]} (\sigma)  \diff \sigma
		\le %
		h \delta (x) ^\gamma.
		\end{equation*}
		Hence
		\begin{equation*}
		\int_{t_0}^{t_0+h} \int_A |u| \delta^\gamma
		\le  C\int_{t_0}^{t_0+h} \int_\Omega |u| \varphi_1 \le
		C h  \left( 	\int_\Omega |u_0 (x)| \delta^\gamma \diff x + \int_\Omega \int_0^T | f(t, x) | \delta^\gamma \diff t \diff x \right) .
		\end{equation*}
		
		\noindent \textbf{Step 2. Space compactness.} We take $\phi(t,x) = \chi_A(x)\varphi_1(x)$ in the weak-dual formulation.
		Applying an argument similar to \cite{abatangelo+gc+vazquez2019} one can show that, for some $p>1$ small
		\begin{equation*}
		\sup_{x \in \Omega}  \int_\Omega \left( \frac{ \G(x,y) \delta(y)^\gamma}{ \delta(x)^\gamma } \right)^p \diff y
  \leq C  < \infty.
		\end{equation*}
		Therefore, for some $q< \infty$ large, $\G: \delta^\gamma L^q (\Omega) \to \delta^\gamma L^\infty (\Omega)$. In particular
		\begin{equation*}
		\left \| \frac{\Green[\chi_A \varphi_1]}{\delta^\gamma} \right \|_{L^\infty(\Omega)}
\le  C  \left \| \frac{\chi_A \varphi_1}{\delta^\gamma} \right \|_{L^q(\Omega)}
\le   C  \left \| \frac{\varphi_1}{\delta^\gamma} \right \|_{L^\infty(\Omega)} |A|^{\frac 1 q}.
		\end{equation*}
		Using \eqref{eq:S submarkovian} and the fact that $\G(x,y) = \int_0^\infty \p (t,x,y) \diff t$ we have
		\begin{align*}
		\Heat[0, \phi ,0](t,x)
		&= \int_0^t \Semi(t-\sigma) [  \chi_A  \varphi_1 ] (x) \diff \sigma \\
		&= \int_0^t \Semi(\sigma) [\chi_A\varphi_1] (x) \diff \sigma \\
		&\le \int_0^\infty \Semi(\sigma) [  \chi_A \varphi_1  ] (x) \diff \sigma \\
		&= \Green [\chi_A\varphi_1] (x) \\
		&\le \omega (|A|)  \delta(x) ^\gamma .
		\end{align*}
		Thus
		\begin{equation*}
		\int_0^T \int_A |u| \delta^\gamma
		\leq C\int_0^T \int_A |u|\varphi_1
		\le  \omega_{K} (|A|) \left( 	\int_\Omega |u_0 (x)| \delta^\gamma \diff x + \int_\Omega \int_0^T | f(t, x) | \delta^\gamma \diff t \diff x \right) .
		\end{equation*}
		
		\noindent \textbf{Step 3. Space-time compactness.}
		We write
		\begin{align*}
		\int_{t_0}^{t_0+h}\int_A |u| \delta^\gamma
		&\le  \int_0^T \int_A
		\left( \chi_{[t_0,t_0+h]}(t)  |u(t,x)|^{\frac 1 2} \delta^{\frac{\gamma}{2}}
		\right)
		\left( \chi_A (x)  |u(t,x)|^\frac 1 2 \delta^{\frac{\gamma}{2}}
		\right)\\
		& \le  \left(
		\int_{t_0}^{t_0+h} \int_A |u| \delta^\gamma
		\right)^{\frac 1 2}
		\left( 	\int_0^T \int_A |u| \delta^\gamma
		\right) ^{\frac 1 2}.
		\end{align*}
		Using the time compactness on the first term we recover and the space compactness on the second, we recover the result.
		
		\noindent \textbf{Step 4. Space-time compactness for signed data.}
		In the general case, we split
		\[
		u=\Heat[(u_0)_+,f_+,0]-\Heat[(u_0)_-,f_-,0].
		\]
		and apply Step 3 to each summand.
	\end{proof}

	\section{Weak formulation}
	\label{sec:weak formulation}
	First, let us do some formal computations, which will made fully rigorous below.
	{Integrating} the equation with a test function $\varphi \in {C^\infty(0,T; C_c^\infty(\Omega))}$ we have
	\begin{equation*}
		\int_\Omega \int_0^T u_t (t, x) \varphi (t, x) \diff t \diff x + \int_\Omega \int_0^T \Ls u (t, x) \varphi (t, x) \diff t  \diff x = 		\int_\Omega \int_0^T f(t, x) \varphi (t, x) \diff t  \diff x.
	\end{equation*}
	Integrating by parts in the equation through \eqref{eq:integration by parts} we have
	\begin{multline}
	\label{eq:weak}
   \int_\Omega u(T,x) \varphi(T,x) \diff x
   + \int_\Omega \int_0^T u (t, x) \left( - \varphi_t (t, x) + \Ls \varphi (t, x) \right) \diff t  \diff x  \\
    = 	\int_\Omega u_0 (x) \varphi (0,x) \diff x
    + \int_\Omega \int_0^T f(t, x) \varphi (t, x) \diff t \diff x
    + \int_0^T \int_{\partial \Omega} h(t, \zeta) \D_\gamma \varphi (t, \zeta) \diff t  \diff \zeta.
	\end{multline}

	\subsection{Weak-dual formulation}
	Let $\psi \in L^\infty((0,T) \times \Omega)$. We define $\overline \varphi \defeq \Heat [0, \phi, 0]$ is a solution of
	\begin{equation*}
		\begin{cases}
		\frac{\partial \overline \varphi}{\partial t} + \Ls \overline \varphi = \phi(t, x), & \text{ in } (0,T) \times \Omega, \\
		 \overline \varphi (t,x)  = 0, & \text{ in }(0,T) \times \Omega^c,  \\
		 \overline \varphi (0,x) = 0, & \text{ in }\Omega.
		\end{cases}
	\end{equation*}
	Then $\varphi (t,x) = \overline \varphi (T-t, x)$ satisfies
	\begin{equation*}
	\begin{cases}
	-\frac{\partial  \varphi}{\partial t} + \Ls  \varphi = \phi(T-t, x), & \text{ in }(0,T) \times \Omega, \\
	 \varphi (t,x)  = 0, & \text{ in }(0,T) \times \Omega^c, \\
	 \varphi (T,x) = 0, & \text{ in }\Omega.
	\end{cases}
	\end{equation*}
	Taking this as a test function we have the formulation
	\begin{multline}
	\label{eq:weak phi}
	 \int_\Omega \int_0^T  %
u (t, x) \phi(T-t,x) \diff t  \diff x
\\
	= 	\int_\Omega u_0 (x) \Heat[0,\phi,0] (T, x) \diff x %
+ \int_\Omega \int_0^T f(t, x)  \Heat[0,\phi,0] (T-t, x) \diff t \diff x
\\
+ \int_0^T \int_{\partial \Omega} h(t, \zeta) \D_\gamma  \Heat[0,\phi,0] (T-t, \zeta) \diff t  \diff \zeta, %
\quad \forall \phi \in \delta^\gamma L^\infty((0,T)\times \Omega).
	\end{multline}
Note that by \Cref{thm:continuity of H}, $\Heat[0,\phi,0](T,\cdot)\in \delta^\gamma L^\infty(\Omega)$, $\Heat[0,\phi,0](T-\cdot,\cdot)\in \delta^\gamma L^\infty((0,T)\times \Omega)$, and also $\D_\gamma\Heat[0,\phi,0](T-\cdot,\cdot)\in  L^\infty((0,T)\times\Omega)$.

	\subsection{An alternative formulation with \texorpdfstring{$\Green$}{G}}

	There is an alternative weak-dual formulation in analogy to the elliptic case.
	Letting $\varphi(t,x) = \Green[\psi (t, \cdot)] (x)$ and $\psi (T,x) = 0$ we can write
    \begin{equation}
		\label{eq:weak psi}
	\int_\Omega \int_0^T u
    \left( - \Green \left[ \frac{\partial \psi}{\partial t} \right] + \psi \right)
=  \int_\Omega u_0  \Green[\psi(0, \cdot)] + \int_\Omega \int_0^T f \Green [\psi]
+ \int_0^T \int_{\partial \Omega} h  \D_\gamma \Green[\psi],
	\end{equation}
for all $\psi\in W^{1,\infty}(0,T;\delta^\gamma L^\infty(\Omega))$  such that $
    \psi(T,\cdot)\equiv 0.$
	The two formulations above are related by
	\begin{equation}
		\label{eq:relation psi and phi}
		\Green[\psi(t,\cdot) ] (x) = \Heat[0,\phi,0] (T-t,x).
	\end{equation}
	
	Because the test function space is smaller, this formulation is weaker (i.\,e., less strict) than the previous one.
	\begin{theorem}
		\label{thm:weak form equivalence}
		Let $u \in L^1 (0,T; L^1 (\Omega, \delta^\gamma))$
		satisfy \eqref{eq:weak phi}.
		Then $u$ satisfies \eqref{eq:weak psi}. %
	\end{theorem}

	\begin{proof}
		The crux of the proof is to show that if $\psi \in W^{1,\infty} ( 0,T ; \delta^\gamma L^\infty ( \Omega))$ is such that $ \psi(T,\cdot)\equiv 0$, then 	
		\begin{equation*}
		\phi(t,x) \defeq - \Green \left[ \frac{\partial \psi}{\partial t} (T-t, \cdot) \right] (x) + \psi ( T -  t,x)
		\end{equation*}
is a valid test function in \eqref{eq:weak phi} and satisfies \eqref{eq:relation psi and phi}. Indeed, since  $\frac{\partial \psi}{\partial t} (t, \cdot),\psi(t, \cdot) \in \delta^\gamma L^\infty (\Omega)$  uniformly in $t$,
so is $\phi(t,\cdot)$.
To show \eqref{eq:relation psi and phi}, we compute
		\begin{align*}
		\left \langle 	\Heat \left[  0, \phi ,0   \right] (t,\cdot)  , \varphi_k \right \rangle
        &=
		\int_0^t e^{-\lambda_k (t - \sigma)} \langle \phi (\sigma, \cdot) , \varphi_k \rangle \diff \sigma \\
		&=\int_{T-t}^{T} e^{-\lambda_k ( t +\sigma - T)} \langle \phi  (T-\sigma, \cdot) , \varphi_k \rangle \diff \sigma \\
		&=  \int_{T-t}^T  e^{-\lambda_k ( t +\sigma - T)} \left \langle - \Green \left[ \frac{\partial \psi}{\partial t} \right] (\sigma, \cdot) + \psi ( \sigma, \cdot) , \varphi_k \right \rangle \diff \sigma \\
		&=  \int_{T-t}^T  e^{-\lambda_k ( t +\sigma - T)} \left[   - \lambda_k^{-1}  \frac{\partial }{\partial t} \langle \psi (\sigma, \cdot), \varphi_k \rangle + \langle \psi (\sigma , \cdot) , \varphi_k  \rangle \right]  \diff \sigma \\
		&=  \lambda_k^{-1} (\langle \psi (T-t, \cdot) , \varphi_k  \rangle - e^{-\lambda_k t} \langle \psi (T, \cdot) , \varphi_k  \rangle) \\
		&=  \lambda_k^{-1} \langle \psi (T-t, \cdot) , \varphi_k  \rangle .
		\end{align*}
		This completes the proof.
	\end{proof}
	\begin{remark}
		For the other implication, we would try to construct
		\begin{equation*}
			\psi (t,x) = - A[ \Heat [0,\phi,0] (T-t, \cdot)]] (x).
		\end{equation*}
		However, showing the existence of weak time-derivative is more delicate.
	\end{remark}

	\subsection{Uniqueness and estimates of the weak-dual solution}

	Formulation \eqref{eq:weak phi} is good for estimates.
\begin{theorem}
\label{thm:uniqueness}
Let $u\in L^1(0,T;L^1(\Omega,\delta^\gamma))$  satisfy \eqref{eq:weak phi} for $u_0 \in L^1(\Omega,\delta^\gamma)$, $f\in L^1(0,T;L^1(\Omega,\delta^\gamma))$ and $h\in L^1((0,T)\times\partial\Omega)$. Then, it holds that
	\begin{equation}
	\label{eq:L1 estimate}
	\begin{split}
	\int_0^T \int_\Omega
	| u | \delta^\gamma
	\diff x \diff t
	&\leq C \left( \int_\Omega |u_0|\delta^\gamma +
	\int_{0}^{T}
	\int_{\Omega}
	|f| \delta ^\gamma
	\diff x
	\diff t  + \int_0^T \int_{\partial \Omega} |h  | \diff t \diff \zeta  \right).
	\end{split}\end{equation}
In particular, there exists at most one solution $u$ of \eqref{eq:weak phi}. Moreover, if $u_0,f,h\geq 0$, then $u\geq0$.
\end{theorem}

\begin{proof}
	We take $\phi(t,x) = \sign(u(t,x)) \varphi_1(x)$,
	where we recall that $\varphi_1 \asymp \delta^\gamma$.
    By \Cref{thm:continuity of H}, \eqref{eq:L1 estimate} holds for any solution $u$ of \eqref{eq:weak phi}. To prove that non-negativity is preserved, take $\phi(t,x)=\sign_-(u(t,x))$.
\end{proof}

	\section{Proof of \texorpdfstring{\Cref{thm:main}}{Theorem 1.1}}
	\label{sec:existence}

	The aim of this section is to show that the candidate solution $\Heat[u_0,f,h]$ given by \eqref{eq:general form of the solution} is in fact the unique solution of the weak-dual formulation \eqref{eq:weak phi}. Since we can write
	\begin{equation*}
		\Heat[u_0,f,h] = \Heat[u_0,f,0] + 	\Heat[0,0,h],
	\end{equation*}
	we will check the the  parts separately.

\subsection{Theory without singular boundary data: \texorpdfstring{$h = 0$}{h=0}}

	We propose to call that part the standard theory. We have the following result
	
		\begin{theorem}
			\label{thm:main h = 0}
		Let $f \in L^1 (0,T; L^1 (\Omega, \delta^\gamma))$ and $u_0 \in L^1 (\Omega, \delta^\gamma)$. Then, $u = \Heat[u_0,f,0] \in L^1 (0,T;L^1 (\Omega, \delta^\gamma) )$
        is the unique solution of \eqref{eq:weak phi}. %
	Furthermore, \eqref{eq:L1 estimate} holds.
	\end{theorem}
\begin{proof}
		\noindent \textbf{Step 1. Assume $u_0 \in L^2 (\Omega)$ and $f \in L^2 ((0,T) \times \Omega)$.}
	
	In this setting can write
	\begin{align*}
		u_0(x) = \sum_{{m}=1}^\infty \langle u_0, \varphi_m \rangle \varphi_m(x), \qquad
		f(t,x)  = \sum_{{m}=1}^\infty \langle f(t, \cdot), \varphi_m \rangle \varphi_m(x), \qquad
		\phi(t,x) = \sum_{{m}=1}^\infty \langle \phi (t, \cdot), \varphi_m \rangle \varphi_m(x),
	\end{align*}
	where the coefficients $\angles{u_0, \varphi_m}$ are $L^2$-summable and
	\begin{equation*}
		\sum_{m=1}^\infty \int_0^T \angles {f(t, \cdot) , \varphi_m} ^2 \diff t < \infty , \qquad  \sum_{m=1}^\infty \int_0^T \angles {\phi (t, \cdot) , \varphi_m} ^2 \diff t  < \infty.
	\end{equation*}
	Since we know there exists a unique solution due to \Cref{thm:uniqueness}, it is enough to show that our candidate satisfies the equation. Applying \eqref{thm:continuity of H} we know that $u \in L^2 ((0,T)\times \Omega) = L^2 ((0,T); L^2 (\Omega))$ and hence it is sufficient to work on the coefficients of the eigendecomposition.
	Therefore, we can write
	\begin{align*}
    \langle u(t,\cdot),  \varphi_{m} \rangle
=\angles{\Heat[u_0,f,0](t,\cdot),\varphi_m}
&=  e^{-\lambda_m t} \langle u_0, \varphi_m \rangle +  \left( \int_0^t e^{-\lambda_m (t-\sigma)} \langle f(\sigma, \cdot), \varphi_m \rangle \diff \sigma \right) \\	
		\langle \Heat [0, \phi ,0] (t,\cdot)  , \varphi_m \rangle
		&
		=  \left( \int_0^t e^{-\lambda_m (t-\sigma)} \langle \phi(\sigma, \cdot), \varphi_m \rangle \diff \sigma \right).
	\end{align*}
	For functions in $L^2 (\Omega)$ we have that
	\begin{equation*}
		\int_\Omega \left(  \sum_{k=1}^\infty a_k \varphi_k (x)  \right) \left( \sum_{m=1}^\infty b_m \varphi_m (x) \right) \diff x = \sum_{k,m = 1} ^\infty a_k b_m \int_\Omega \varphi_k (x) \varphi_m (x) \diff x = \sum_{m=1} ^\infty a_m b_m.
	\end{equation*}
	Therefore, in order to produce the left hand side of \eqref{eq:weak phi}, it suffices to %
integrate both sides in $t$
against the component of the test function $\phi$
in the same index $m$.
	Applying Fubini's theorem, we have that
		\begin{align*}
			\int_0^T  \langle u (t, \cdot), \varphi_{m} \rangle &  \langle  \phi(T-t,\cdot ),  \varphi_m \rangle  \diff t  \
			= 	\langle u_0 , \varphi_{m} \rangle  \left( \int_0^T e^{-\lambda_m (T-\sigma)} \langle \phi(\sigma, \cdot), \varphi_m \rangle \diff \sigma \right)\\
			& + \int_0^T \langle f (t, \cdot) , \varphi_{m} \rangle \left( \int_0^{T-t} e^{-\lambda_m (T-t-\sigma)} \langle \phi(\sigma, \cdot), \varphi_m \rangle \diff \sigma \right) \diff t .
	\end{align*}
	Therefore, summing up in $m$ yields
	\begin{equation*}
	\int_\Omega \int_0^T  u (t, x) \phi(T-t,x) \diff t  \diff x
	= 	\int_\Omega u_0 (x) \Heat[0,\phi,0] (T, x) \diff x
	+ \int_\Omega \int_0^T f(t, x)  \Heat[0,\phi,0] (T-t, x) \diff t \diff x.
	\end{equation*}

	\noindent \textbf{Step 2. General setting.}
		 Let $u_{0,k} \in L^2(\Omega)$ and $f_k \in L^2 ((0,T)\times \Omega)$ converge to $u_0$ and $f$ in the corresponding weighted $L^1$ spaces. Due to the continuity of $\Heat$, we have that $\Heat[u_{0,k},f_k,0] \to \Heat[u_0, f, 0]$ in $L^1 (\Omega, \delta^\gamma)$. This allows us to pass to the limit in \eqref{eq:weak phi}. The uniqueness is already known from \eqref{thm:uniqueness}.
	\end{proof}

\begin{remark}
	As mentioned in \Cref{rem:energy setting}, $ \Semi(t)$ is also defined in the energy spaces. Naturally, this is extended to $ \Heat[u_0,f,0](t,\,\cdot\,)$ as long as $u_0 \in H_\Ls^{-1} (\Omega)$ and $f \in L^1 (0,T; H_\Ls^{-1} (\Omega))$. The notion of weak solution needs to be extended by introducing duality products.
\end{remark}

\subsection{Theory for singular boundary data: \texorpdfstring{$h \ne 0$}{h=0}}

 We are finally ready to give a linear theory for singular boundary data. The argument is as follows: we will construct a solution by an interior approximation procedure, as in the elliptic setting. Due to \Cref{thm:uniqueness} this is the unique solution. Then, by applying Fubini we show that $u = \Heat[0,0,h]$. 	

	\begin{theorem}
		\label{thm:main h ne 0}
		Let $h \in L^1 (0,T; L^1 (\partial \Omega))$ and $u_0, f = 0$. Then $u = \Heat[0,0,h]$ given by \eqref{eq:general form of the solution} is the unique function in $L^1 (0,T; L^1 (\Omega, \delta^\gamma))$ such that \eqref{eq:weak phi} holds.
	\end{theorem}

	\begin{proof}
\noindent {\bf Step 1: Both $h \in C^\infty([0,T] \times \partial \Omega)$ and $\phi \in \delta^\gamma L^\infty((0,T)\times\Omega)$ are non-negative.}
		Take $f_j$ as in \eqref{eq:concentration to boundary f_j}. Let $u_j \defeq \Heat[0,f_j,0]$.
		By \Cref{lem:uniform integrability} and Dunford--Pettis Theorem the sequence $u_j \delta^\gamma $ is precompact in $L^1( (0,T) \times \Omega )$ and hence, there exists $u$ such that up to a subsequence
		\begin{equation*}
		u_j \delta^\gamma \rightharpoonup u \delta^\gamma \quad \text{ in } L^1( (0,T) \times \Omega ).
		\end{equation*}
		
		Due to \Cref{thm:main h = 0}, we can write
		\begin{equation*}
		\begin{split}
		\int_\Omega \int_0^T u_j (t,x) \phi(T-t,x)   =   \int_0^T \frac{|\partial \Omega|}{|A_j|}\int_{A_j} h(t,P_{\partial \Omega } (x)) \frac {\Heat[0,\phi,0](T-t,x)}{\delta(x)^\gamma} \diff x \diff t , \\
		\end{split}
		\end{equation*}
		As in \cite{abatangelo+gc+vazquez2019}, it is easy to prove by using the tubular neighbourhood theorem that, for every $t \in (0,T)$,
		\begin{equation*}
		F_j (t) \defeq
		\frac{|\partial \Omega|}{|A_j|}\int_{A_j} h(t,P_{\partial \Omega } (x)) \frac {\Heat[0,\phi,0](T-t,x)}{\delta(x)^\gamma} \diff x
    \to
    \int_{\partial \Omega} h ( t , \zeta) \D_\gamma  \Heat[0,\phi, 0]  (T-t,\zeta) \diff \zeta = \vcentcolon F(t).
		\end{equation*}
		Hence, we have that
		\begin{equation*}
			F_j (t) \le C \left\| \frac { \Heat[0,\phi,0](T-t,\cdot)}{ \delta^\gamma} \right\|_{L^\infty (\Omega)} \int_{\partial \Omega} h(t, \zeta)  \diff \zeta .
		\end{equation*}
		Therefore, by the Dominated Convergence Theorem
		\begin{equation*}
			\int_0^T F_j (t) \diff t \to \int_0^T F(t) \diff t.
		\end{equation*}
		Let us now show that $u = \Heat[0,0,h]$ in the weak-dual sense \eqref{eq:weak phi}. We write the kernel expression and apply Fubini's theorem
		for non-negative functions
		
		to deduce that
		\begin{align*}
		\int_\Omega \int_0^T u(t,x)  \phi(T-t,x) \diff t \diff x &=   \int_0^T \int_{\partial \Omega} h(t, \zeta)  \D_\gamma \Heat[0,\phi,0](T-t,\zeta) \diff \zeta \diff t \\
		&=  \int_0^T \int_{\partial \Omega} h(t, \zeta)  \left( \int_0^{T-t} \int_\Omega \D_\gamma \p (T-t-\sigma,\zeta,y) \phi(\sigma, y) \diff \sigma \diff y \right) \diff \zeta \diff t \\
		&= \int_0^T \Heat[0,0,h] (t,x)  \phi(T-t,x) \diff t \diff x.
		\end{align*}

\noindent {\bf Step 2: Both $h \in C^\infty([0,T]\times\partial\Omega)$ and $\phi$ are sign-changing.}
		When $h$ and $\phi$ change sign, we can decompose them into their positive and negative parts, and apply {\bf Step 1}.
\medskip

\noindent {\bf Step 3: General case $h \in L^1((0,T)\times\partial\Omega)$.} We apply {\bf Step 2} to an approximating sequence for $h\in L^1 ((0,T) \times \partial \Omega)$, namely $h_j \in C^\infty([0,T]\times\partial\Omega)$ with $h_j\to h$ in $L^1 ((0,T) \times \partial \Omega)$.
		Since $\Heat$ and the weak formulation allow us to pass to the limit $j\to\infty$, we do so as in \Cref{thm:main h = 0}.
	\end{proof}

		\begin{remark}
		We present a formal computation to show that $\Heat[0,0,h]$ satisfies the weak formulation \eqref{eq:weak psi}. Assume that $h$ is regular.
		Fix a test function $\psi$. We define $u_k (t,x) \defeq \langle u(t,\cdot), \varphi_k \rangle$. Notice that
		\begin{align*}
		u_k (t,x) &=   \int_0^t e^{-\lambda_k (t-\sigma)}\left( \int_{\partial \Omega}  h(\sigma,\zeta ) \D_\gamma  {\varphi_k (\zeta)} \diff \zeta \right) \diff \sigma.
		\end{align*}
		Define also
		\begin{equation*}
		c_k (t) = \langle \psi_k(t, \cdot), \varphi_k \rangle .
		\end{equation*}
		Since $\psi(T,x) = 0$ we have that $c_k(T) = 0$. Clearly
		\begin{align*}
		\left \langle - \Green \left[ \frac{\partial \psi}{\partial t} (t, \cdot) \right] + \psi (t, \cdot) , \varphi_k \right \rangle =  -\lambda_k^{-1} c_k'(t)  + c_k (t)
		\end{align*}
		Hence
		\begin{align*}
		\int_\Omega \int_0^T &u _k(t,x) \left( - \Green \left[ \frac{\partial \psi}{\partial t}(t, \cdot) \right](x) + \psi (x) \right) \diff t \diff x\\
		&=   \left\{ \int_0^T \left(   \int_0^t e^{-\lambda_k (t-\sigma)}\left( \int_{\partial \Omega}  h(\sigma,\zeta ) \D_\gamma  {\varphi_k (\zeta)} \diff \zeta \right) \diff \sigma  \right) \left( -\lambda_k^{-1} c_k'(t)  + c_k (t) \right) \diff t \right \} \int_\Omega \varphi_k(x)^2 \diff x \\
		&=   \int_0^T  \int_0^t e^{-\lambda_k (t-\sigma)} \left( -\lambda_k^{-1} c_k'(t)  + c_k (t) \right) \left( \int_{\partial \Omega}  h(\sigma,\zeta ) \D_\gamma  {\varphi_k (\zeta)} \diff \zeta \right) \diff \sigma   \diff t\\
		&= \int_0^T \left( \int_t^T e^{-\lambda_k (t-\sigma)} \left( -\lambda_k^{-1} c_k'(t)  + c_k (t) \right)   \diff t \right) \left( \int_{\partial \Omega}  h(\sigma,\zeta ) \D_\gamma  {\varphi_k (\zeta)} \diff \zeta \right) \diff \sigma\\
		&= \int_0^T (- \lambda_k) c_k(t)   \left( \int_{\partial \Omega}  h(\sigma,\zeta ) \D_\gamma  {\varphi_k (\zeta)} \diff \zeta \right) \diff \sigma\\
		&= \int_0^T  \int_{\partial \Omega}  h(\sigma,\zeta ) \D_\gamma  \Green[  c_k (t) \varphi_k (t) ] \diff \zeta \diff \sigma.\\
		\end{align*}
		Therefore
		\begin{align*}
		\int_\Omega \int_0^T &\left( \sum_{k=1}^m \langle u(t, \cdot), \varphi_k \rangle \varphi_k (x) \right)  \left( - \Green \left[ \frac{\partial \psi}{\partial t}(t, \cdot) \right](x) + \psi (x) \right) \diff t \diff x\\
		&= \int_0^T  \int_{\partial \Omega}  h(\sigma,\zeta ) \D_\gamma  \Green\left[ \sum_{k=1}^m \langle \psi(t, \cdot), \varphi_k \rangle \varphi_k  \right] (\zeta)   \diff \zeta \diff \sigma.
		\end{align*}
	\end{remark}
	\subsection{General setting}
	\label{sec:proof main result general}
	Due to the linearity, joining \Cref{thm:uniqueness}, \Cref{thm:main h = 0} and \Cref{thm:main h ne 0} we recover the general which is equivalent to \Cref{thm:main}.
	\begin{theorem}
		We have that
	$$
		\Heat : L^1 (\Omega, \delta^\gamma ) \times L^1 (0,T;L^1 (\Omega, \delta^\gamma ) ) \times L^1 ( (0,T) \times \partial \Omega ) \longrightarrow L^1(0,T; L^1 (\Omega, \delta^\gamma ) )
	$$
		is continuous and we have that estimate \eqref{eq:L1 estimate}. For $(u_0, f,h)$ in this domain, $u = \Heat[u_0, f,h]$ the unique function in $L^1 (0,T; L^1 (\Omega, \delta^\gamma))$ such that \eqref{eq:weak phi} holds.
	\end{theorem}

	\section{Agreement between the elliptic and parabolic theories}
	\label{sec:elliptic and parabolic}
	Here we prove that the solution of the parabolic problem with time-independent data converges to the solution of the elliptic problem as $t \to +\infty$.
	\begin{theorem}
		Let $u_0 \in L^1 (\Omega, \delta^\gamma)$, $f(t, x) = f(x) \in L^1 (\Omega, \delta^\gamma)$ and $h(t, x) = h(x) \in L^1 (\partial \Omega)$ then
		\begin{equation*}
			\Heat[u_0,f,h] (t,\cdot) \to \Green[f] + \Martin[h] \qquad \text{ in } L^1 (\Omega, \delta^\gamma) \text{ as } t \to +\infty.
		\end{equation*}
	\end{theorem}
	\begin{proof}
		Due to the weak-dual formulation, it is immediate to see that $\Heat[\Green[f],f,0] = \Green[f]$ and $\Heat[\Martin[h],0,h] = \Martin[h]$. Hence,
		\begin{equation*}
			\Heat[u_0, f,h] - \Green[f] - \Martin[h] = 	\Heat[u_0 - \Green[f] - \Martin[h] ,0,0]  =   \Semi(t) \left[ u_0 - \Green[f] - \Martin[h]  \right].
		\end{equation*}
		From \Cref{prop:decay of L1 delta} we have that
		\begin{equation*}
		 \|  \delta^\gamma \Semi(t) \left[ u_0 - \Green[f] - \Martin[h]  \right] 	  \|_{L^1 (\Omega)} \to 0,
		\end{equation*}
		and this concludes the proof.
	\end{proof}

	\section{Main examples of operators \texorpdfstring{$\Ls$}{L}}
\label{sec:examples}
\subsection{Restricted fractional Laplacian (RFL)}

	This is the main example as far as the literature is concerned. The operator is given by
	\begin{equation*}
		\RFL u(x) = \text{PV} \int_{\mathbb R^d} \frac{u(x) -u(y)}{|x-y|^{n+2s}} \diff y\,,
	\end{equation*}
	where $u$ is extended by $0$ outside $\Omega$. There is a large literature. Thus, hypotheses  \eqref{eq:G estimate}, \eqref{eq:G delta Linf to delta C}, and \eqref{eq:S submarkovian} were already checked in \cite{bonforte+figalli+vazquez2018}, and  \eqref{eq:Green normal} %
in \cite{abatangelo+gc+vazquez2019}, and it is proved that the corresponding exponent is $\gamma = s \in (0,1)$. Moreover,
from \cite{Chen2010, Bogdan2010}, we have for any $T > 0$,
	\begin{equation}\label{eq:RFL-p-bd}
		\p (t,x,y) \overset T \asymp
		\begin{dcases}
			t^{-\frac{n}{2s}} \left( 1 \wedge \frac{t^{\frac 1 {2s}}}{|x-y|} \right)^{n + 2s} \left( 1 \wedge \frac{\delta(x)}{t^{\frac{1}{2s}}} \right)^{s } \left( 1 \wedge \frac{\delta(y)}{t^{\frac{1}{2s}}} \right)^{s },  			& t < T, \\
			e^{-\lambda_1 t} \delta(x)^{s} , \delta(y)^s,  & t \ge T,
		\end{dcases}
	\end{equation}
	where $\lambda_1$ is the smallest  eigenvalue of $\RFL$.  (Hereafter, $\overset T \asymp$ means that the constant in the
equivalence  may depend on $T$.)
	This proves  \eqref{eq:S boundary regularity}.
	
	Concerning regularity, in \cite{Fdz-Real+RosOton2016} the authors showed that
	\begin{equation*}
	\left\| 	\frac{\Semi (t) u_0 }{\delta ^s} \right\|_{C^{s-\varepsilon} (\overline \Omega)} \le C(t_0) \| u _0 \|_{L^2 (\Omega)}, \qquad \forall t \ge t_0.
	\end{equation*}
	This guarantees \eqref{eq:S boundary regularity}. Their argument is based on Weyl's law, which is already known \cite{BlumGetoor1959,Geisinger2014,Grubb2015}.
	We provide some further information. Since the $s$-normal derivative of the heat kernel is known to exist by \Cref{thm:S boundary reg}, it satisfies the two-sided estimate  %
	\[
    \D_s\p(t,\zeta, y)
	\overset T \asymp
	\begin{dcases}
	t^{ -\frac {n+s}{2s}} \left( 1 \wedge \frac{t^{\frac 1 {2s}}}{|\zeta-y|} \right)^{ n + 2s} \left( 1 \wedge \frac{\delta(y)}{t^{\frac{1}{2s}}} \right)^{s },  			& t < T, \\
	e^{-\lambda_1 t} \delta(y)^{s},
	& t \ge T.
	\end{dcases}
	\]

Furthermore, we have the following	

    \begin{lemma}[Regularisation of RFL heat semigroup]
    For any $t>0$, the RFL semigroup map $\Semi(t)$ is continuous from $L^1(\Omega,\delta^s)  \to \delta^s C^{s-\varepsilon}(\overline\Omega)$.  Furthermore, $L^1 (\Omega, \delta^s)$ is the largest set of admissible integrable data in the sense that, if $u_0 \ge 0$ then we have the Hopf inequality
    \begin{equation*}
    	\frac{ \Semi(t) u_0 (x) } {\delta(x)^s} \asympT \int_\Omega u_0 (y) \delta(y)^s \diff y.
    \end{equation*}
    \end{lemma}

    \begin{proof}
    	Take $T = t / 2$ in \eqref{eq:RFL-p-bd}. Then, we have that
    	\begin{equation*}
    	\left| 	\frac{\Semi (t) [u_0] (x)}{\delta (x)^s} \right| \le C(t) e^{-\lambda_1 t} \int_\Omega |u_0(y)| \delta(y)^s \diff y,
        \qquad  \forall u_0\in L^1(\Omega,\delta^s).
    	\end{equation*}
    	For non-negative $u_0$ we take advantage of the lower bound of the kernel.
    \end{proof}

	\subsection{Spectral fractional Laplacian (SFL)}
	Let $\lambda_k [-\Delta]$ be the eigenvalues of the usual Laplacian, with eigenfunctions $\varphi_k$. Then for any $s\in(0,1)$ we define
	\begin{equation*}
		\SFL u(x) = \sum_{k=1}^\infty  (\lambda_k[-\Delta])^s  \langle u , \varphi_k \rangle \varphi_k (x),
	\end{equation*}
 whenever the right hand side converges.
	Hypotheses  \eqref{eq:G estimate}, \eqref{eq:G delta Linf to delta C}, and \eqref{eq:S submarkovian}, were already shown to hold in \cite{bonforte+figalli+vazquez2018}, and  \eqref{eq:Green normal} was checked in \cite{abatangelo+gc+vazquez2019}, with
 parameter $\gamma = 1$.  These facts  come from the previous literature.
	
	There are different ways to verifying \eqref{eq:S boundary regularity} in this setting. First, through the eigenvalues of the SFL are the $s$-power of those of the usual Laplacian,  we know that  the energy spaces are 
	\begin{equation*}
	H^k (\Omega) = \left \{  u \in L^2(\Omega) : \sum_{m=1}^\infty  (\lambda_m[-\Delta])^{k}  \langle u, \varphi_m \rangle^2  < +\infty  \right \}.
	\end{equation*}
	Hence $\Semi(t) : L^2 (\Omega) \longrightarrow H^k (\Omega)$ for any  $k>0$. By Sobolev embedding theorem, we also have $\Semi(t): L^2(\Omega) \longrightarrow \delta C(\overline \Omega)$.
	
	Another way is through the estimates.  From  \cite[Theorems 3.1 and 3.9]{Song2004} (see also \cite{Song2017}) we have an additional exponential correction (which is only relevant as $t \to 0$): for any $T > 0$ and $ 0 < t \le T $,

	\begin{multline*}
	\nonumber
	C_1(T)\left( \frac{ \delta(x) \delta(y) }{t} \wedge 1 \right)t^{-\frac{n}{2}} \exp \left( - \frac {C_2(T) |x-y|^2 }{ t} \right)
\\
	\le \p(t,x,y) \\
\le
	\nonumber
	C \left( \frac{ \delta(x) \delta(y) }{t} \wedge 1 \right)t^{-\frac{n}{2}} \exp \left( - \frac { |x-y|^2 }{6t} \right).
	\end{multline*}

 Note that these estimates are independent of the fractional order $2s$. 
	Thus

	\begin{gather*}
	C_1(T) \delta(y) t^{-\frac{n+2}{2}} \exp \left( - \frac {C_2(T) |\zeta -y|^2 }{ t} \right)
	\le \D_1 \p(t,\zeta,y) \le
	C{\delta(y)}t^{-\frac{n+2}{2}} \exp \left( - \frac { |\zeta -y|^2 }{6t} \right).
	\end{gather*}

	Again, if $u_0 \ge 0$ we have that
	\begin{equation*}
		\frac{ \Semi(t) [u_0 ](x) } {\delta(x)} \asympT \int_\Omega u_0 (y) \delta(y) \diff y.
	\end{equation*}
	Hence, the optimal  class of  data is precisely $L^1 (\Omega, \delta)$.

	\subsection{Censored fractional Laplacian (CFL)}
	We define the operator by the singular integral expression
	\begin{equation*}
				(-\Delta)_\CFL^s  u(x) = \text{PV} \int_{\Omega} \frac{u(x) -u(y)}{|x-y|^{n+2s}} \diff y.
	\end{equation*}
Note that the value of $u$ outside $\overline\Omega$ is irrelevant.
	This operator is also sometimes called regional fractional Laplacian.
Again,  hypotheses \eqref{eq:G estimate}, \eqref{eq:G delta Linf to delta C} and \eqref{eq:S submarkovian} were  already checked in \cite{bonforte+figalli+vazquez2018}, and  \eqref{eq:Green normal} in \cite{abatangelo+gc+vazquez2019}.  The exponents are given by  $s\in(\frac12,1)$, $\gamma = 2s - 1\in(0,1)$.
Recently, it has been shown in \cite{Chen2020} that the case $s \in (0,\frac 1 2]$ does not admit viscosity solution, and this suggests that no natural analogous problem exists in this range.
In \cite{Chen2009} the authors prove there exists a heat kernel  $\p$ that  satisfies the estimates
\begin{equation*}
\p (t,x,y) \asymp
\begin{dcases}
t^{-\frac{n}{2s}} \left( 1 \wedge \frac{t^{\frac{1}{2s}}}{|x-y|} \right)^{n + 2s} \left( 1 \wedge \frac{\delta(x)}{t^{\frac{1}{2s}}} \right)^{2s-1 } \left( 1 \wedge \frac{\delta(y)}{t^{\frac{1}{2s}}} \right)^{2s-1 },  			& t < T, \\
e^{-\lambda_1 t} \delta(x)^{2s-1} \delta(y)^{ 2s-1},  & t \ge T.
\end{dcases}
\end{equation*}
Using it with $T=t/2$, we know in particular that $\p(t,x,y)\overset {t} \asymp \delta(x)^{2s-1}\delta(y)^{2s-1}$, therefore \eqref{eq:S boundary regularity} holds and we have that
	\[
    \D_{2s-1}\p(t,\zeta, y)
	\overset T \asymp
	\begin{dcases}
	t^{-\frac {n+2s-1}{2s}}
    \left( 1 \wedge \frac{t^{\frac 1 {2s}}}{|\zeta-y|}
    \right)^{n + 2s}
    \left( 1 \wedge \frac{\delta(y)}{t^{\frac{1}{2s}}}
    \right)^{2s-1},  			& t < T, \\
	e^{-\lambda_1 t} \delta(y)^{2s-1},
	& t \ge T,
	\end{dcases}
	\]
as well as
	\begin{equation*}
\frac{ \Semi(t) [u_0 ](x) } {\delta(x)^{2s-1}} \asympT \int_\Omega u_0 (y) \delta(y)^{2s-1} \diff y.
\end{equation*}
so the optimal set of data is $L^1 (\Omega, \delta^{2s-1})$.
In this setting, even a one-sided Weyl's law was not known.

\subsection{Exploring new  examples}

Note that the integro-differential operator $\Ls$ can be reconstructed from the Green's function $\G$ whenever \eqref{eq:G estimate} holds. Indeed, if \ $0\leq\G\leq C|x-y|^{-(n-2s)}$, then $\Ls$ can be shown to have a discrete spectrum consisting of a non-decreasing divergent sequence of positive Dirichlet eigenvalues $(\lambda_m)_{m\geq 1}$, with corresponding eigenfunctions $(\varphi_m)_{m\geq 1}$ that form an orthonormal basis of $L^2(\Omega)$ (see for instance \cite[Remark 2.3]{chan+gc+vazquez2020} and \cite{bonforte+figalli+vazquez2018}). Then $\Ls$ can be recovered spectrally as
\[
\Ls u(x)
=\sum_{m=1}^{\infty}
    \lambda_m
    \angles{u,\varphi_m}
    \varphi_m(x),
\qquad\forall u\in H_{\Ls}^1(\Omega).
\]

\begin{itemize}
\item
Under \eqref{eq:G estimate}, $\G$ is not necessarily continuous:
\[
\G(x,y)=|x-y|^{-(n-2s)} \left( \frac{\delta(x) \delta(y)}{|x-y|^2} \wedge 1 \right)^{\gamma}
\left(1+\chi_{A}(x,y)\right),
\]
where $A\subset \Omega\times\Omega$ is any non-empty proper subset.
We do not know if \eqref{eq:G delta Linf to delta C} holds, but it does not seem natural.

\item
When $\G$ is continuous, \eqref{eq:Green normal} does not necessarily hold \cite{abatangelo+gc+vazquez2019}:
\[
\G(x,y)=|x-y|^{-(n-2s)} \left( \frac{\delta(x) \delta(y)}{|x-y|^2} \wedge 1 \right)^{\gamma}
\left(
2+\sin\dfrac{1}{\delta(x)}
\right)
\left(
2+\sin\dfrac{1}{\delta(y)}
\right).
\]
\item
Under all the assumptions \eqref{eq:G estimate}, \eqref{eq:G delta Linf to delta C}, \eqref{eq:S submarkovian} and \eqref{eq:Green normal}, $\Green$, or equivalently $\Semi(t)$, does not necessarily regularizes beyond $\delta^\gamma C(\overline\Omega)$:
\[
\G^{(k)}(x,y)=|x-y|^{-(n-2s)} \left( \frac{\delta(x) \delta(y)}{|x-y|^2} \wedge 1 \right)^{\gamma}
\left(
2+\sin \exp^{(k)} \dfrac{1}{|x-y|^2}
\right).
\]
where $\exp^{(k)}$ denotes the $k$-fold composition of the exponential function. In this case, even the eigenfunctions are not expected to be H\"{o}lder continuous or even continuous in any reasonably quantitative way.
\end{itemize}

\section{Comments, extensions and open problems}

\begin{itemize}

\item A similar theory should hold when $2s\ge  n=1$. In this case, along the diagonal the Green's function is logarithmically singular when $s=\frac12$ and  is regular when $s\in(\frac12,1)$.

\item It seems reasonable to expect that \eqref{eq:Green normal} can be recovered from the rest of the information. The approximations
\begin{equation*}
	\G_k (x,y) = \int_{\frac 1 k}^k \p (t,x,y) \diff t
\end{equation*}	
converge to $\G$ from below and $\G_k (x,y) / \delta(x)^\gamma$ is continuous at the boundary.  It seems reasonable $y$ fixed the convergence should be uniform near the boundary since
\begin{equation*}
	\left \|  \frac{ \G_k (\cdot, y) } {\delta^\gamma } - \frac{ \G (\cdot, y) } {\delta^\gamma }  \right  \|_{L^\infty (U)} \le \int_0^{\frac 1 k} 	\left \|  \frac{ \p (t,\cdot, y) } {\delta^\gamma }  \right  \|_{L^\infty (U)} \diff t + \int_k^{\infty} 	\left \|  \frac{ \p (t,\cdot, y) } {\delta^\gamma }  \right  \|_{L^\infty (U)} \diff t.
\end{equation*}
The second integral is controlled due to \Cref{thm:S boundary reg}. We expect the first to be controlled since for $x \ne y$ formally $\p (0,x,y) = 0$. However, we have not found a rigorous proof of this fact.
	
\item The heat kernel estimates in \Cref{thm:S estimate} for small time is still suboptimal, in comparison to the model examples. According to its proof, one obtains the sharp small-time upper bound away from the diagonal and up to the boundary provided that the following weighted Hardy--Littlewood--Sobolev inequality holds,
    \[
    \int_{\Omega}\int_{\Omega}
        \dfrac{
            \varphi_1(x)
            g(x)
            g(y)
            \varphi_1(y)^{-1}
        }{
            |x-y|^{n-2s}
        }
    \diff y\diff x
    \leq
    \norm[L^{\frac{2n}{n+2s}}(\Omega)]{g},
    \]
    for all $g\in L^{\frac{2n}{n+2s}}(\Omega)$. Its validity would imply
    \[
    \p(t,x,y)
    \leq Ct^{\frac{n}{2s}}\delta(x)^\gamma\delta(y)^\gamma,
        \qquad \forall t>0,\,x,y\in\Omega.
    \]

\item A second approach to heat kernel estimates for small time is the following. When $\Ls$ is sectorial, the contour integral formulation relates the heat kernel to the resolvent and hence to its Green's function. %
    For this, one would relate the hypotheses, mainly \eqref{eq:G estimate}, to  two-sided estimates for the resolvent. Let $\G_\lambda$ be the Green function of $(L+\lambda)^{-1}$. Then $v(x)=v_y(x)=\G_\lambda(x,y)-\G(x,y)$ satisfies
    \[
    \Ls v+\lambda v=-\lambda G(\cdot,y)
    \in L^p(\Omega),
    \]
    for some (small) $p>1$. Thus \cite{chan+gc+vazquez2020} implies the existence of $v$ which is less singular than $G$ and we should get
    \[
    G_\lambda(x,y)\asymp G(x,y)
        \qquad \forall x\neq y\in\Omega.
    \]
    One way to make it rigorous would be to pass to the weak-dual formulation. Notice that another argument is needed to show a suitable decay when $\lambda$ is large.

\item The full Weyl's law \eqref{eq:Weyl full}, with an upper bound as well as the exact asymptotic constant, is yet to be obtained for general operators. In fact, it is not very convenient to use Fourier analytic techniques in our setting where assumptions are mainly made on the Green's function of $\Ls$. A related interesting problem is to generalize the 1956  theorem by Payne, P\'{o}lya and Weinberger on the spectral gaps \cite{Payne-Polya-Weinberger1956} (see also \cite[Chapter III, \S 7]{Schoen1994}): the Dirichlet eigenvalues $(\lambda_k)_{k=1}^{\infty}$ of the classical Laplacian in a bounded domain $\Omega\subset\re^n$ satisfy $\lambda_{k+1}-\lambda_k\leq \frac{4}{nk}\sum_{i=1}^{k}\lambda_i$. An extension of this estimate into the fractional setting would give some sort of control on the eigenvalue growth. 

\item A possible continuation of this work is the study of the $\Ls$ with lower-order terms, including Schr\"{o}dinger operators.

\item  Another continuation of this work is the study of fractional powers of the heat operator, or time-fractional equations. This kind of operators is  not yet much studied, but one can expect that once the regularity properties are known, our framework will provide a linear theory in a general setting.
\end{itemize}

\section*{Acknowledgments}
HC has received funding from the European Research Council under the Grant
Agreement No 721675. The research of  DGC and JLV was partially supported by grant PGC2018-098440-B-I00 from the Ministerio de Ciencia, Innovación y Universidades (MICINN) of the Spanish Government. JLV is an Honorary Professor at Univ.\ Complutense. We would like to thank Gerd Grubb for her valuable comments on the first draft of the manuscript.

\printbibliography
	\end{document}